\documentclass[12pt]{amsart}

\usepackage[matrix,arrow,curve,frame]{xy}
\usepackage{amsmath,amsthm,amssymb,enumerate}
\usepackage{latexsym}
\usepackage{amscd}
\usepackage[colorlinks=false]{hyperref}
\usepackage{euscript}

\newtheorem{thm}[subsection]{Theorem}

\newtheorem{prop}[subsection]{Proposition}

\newtheorem{cor}[subsection]{Corollary}
\newtheorem{lemma}[subsection]{Lemma}

\newtheorem{remark}[subsection]{Remark}

\theoremstyle{definition}

\numberwithin{equation}{section}

\def\bpartial{{\bar\partial}}

\def\cQ{{\bf Q}}
\def\cO{{\cal O}}

\def\cU{{\cal U}}

\def\cA{{\cal A}}

\def\cA{{\mathcal A}}

\def\cO{{\mathcal O}}

\def\cQ{{\mathcal Q}}

\def\cU{{\mathcal U}}

\def\gg{{\mathfrak g}}

\DeclareMathOperator{\Hom}{Hom}

\DeclareMathOperator{\Ric}{Ric}
\DeclareMathOperator{\Spec}{Spec}
\DeclareMathOperator{\Sym}{Sym}
\DeclareMathOperator{\End}{End}

\newfont{\german}{eufm10}

\begin{document}
\pagestyle{plain}

\title{vector bundles induced from jet schemes}

\author{Bailin Song}


\address{School of Mathematical Sciences, University of Science and
Technology of China, Hefei, Anhui 230026, P.R. China}
\email{bailinso@ustc.edu.cn}

\begin{abstract}
A family of holomorphic vector bundles is constructed on a complex manifold $X$. The space of the holomorphic sections of these bundles are calculated in certain cases. As an application, if $X$ is an $N$-dimensional compact K\"ahler manifold with holonomy group $SU(N)$, the space of holomorphic vector fields on its jet scheme $J_m(X)$ is calculated. We also prove that the space of the global sections of the chiral de Rham complex of a K3 surface is the simple $\mathcal N=4$ superconformal vertex algebra with central charge $6$.
  \end{abstract}


\maketitle
\section{Introduction}
On a Ricci flat compact K\"ahler manifold $X$, the holomorphic
sections of the bundle given by tensors of tangent and cotangent bundles are exactly the parallel sections (page 142 of \cite{YB}). So if the holonomy group of $X$ is $G$, the space of the holomorphic sections of the bundle is isomorphic to the space of the $G$ invariants of the fibre.
In this paper, on a complex manifold $X$, given holomorphic vector bundles  $E$ and $F$,  we construct a family of holomorphic vector bundles $A_m(E,F)$, $m\in \mathbb Z_{\geq 0}$.
If $X$ is an $N$-dimensional compact Hermitian manifold with the holonomy group  $SU(N)$, and $E$ and $F$ are sums of copies of the holomorphic tangent and cotangent bundles, we show that the space of the holomorphic sections of $A_m(E,F)$ is  isomorphic to  the space of the $J_m(SL(N,\mathbb C))$ invariants of the fibre of the bundle.  Here $J_m(SL(N,\mathbb C))$ is the $m$th jet scheme of $SL(N,\mathbb C)$ (see Theorem \ref{thm:holo}).

The idea of the construction of $A_m(E,F)$ is from the jet scheme. The $m$th jet scheme  \cite{EM} $J_m(X)$ of an algebraic scheme $X$ over $\mathbb C$ is determined by its functor of points: for every $\mathbb C$-algebra $A$, we have a bijection
\begin{eqnarray*} \Hom(\Spec(A), J_m(X))\cong\Hom(\Spec(A[t]/(t^{m+1})),X).\end{eqnarray*}
There is a canonical projection $\pi_{m}:J_m(X)\to X$. If $X$ is a projective complex manifold, $E$ is a vector bundle over $X$, then
$J_m(E)$ is a vector bundle over $J_m(X)$. If $X$ is Ricci flat, we expect that sections of the bundle given by  tensors of $J_m(TX)$ and $J_m(T^*X)$   are
 $J_m(SL(N,\mathbb C))$ invariants of the fibre. $A_m(E,F)$ is the vector bundle over $X$, such that the sheaf of sections of $A_m(E,F)$ is the push forward of the sheaf of sections of
 $\Sym^* J_m(E)\otimes \wedge^* J_m(F)$ through $\pi_m$.

 There are two applications of Theorem \ref{thm:holo}. The first application is that we can calculate the holomorphic vector fields on $J_m(X)$,
 \begin{thm}\label{thm:TJ} If $X$ is an $N$-dimensional compact K\"ahler manifold with holonomy group $SU(N)$, then the space of holomorphic vector fields on $J_m(X)$ has dimension $m$.
 \end{thm}
The second application is that we can calculate the global sections of the chiral de Rham complex on any K3 surface.
\begin{thm}\label{thm:gch}If $X$ is a K3 surface, the  space of global sections of the chiral de Rham complex of $X$ is the simple $\mathcal N=4$ superconformal vertex algebra\footnote{In the literature, there are small and large $\mathcal N=4$ superconformal algebras. In this paper, it is the small $\mathcal N=4$ superconformal algebra.} with central charge $6$.
\end{thm}
This generalizes our previous result in \cite{S}, which calculates the global sections of the chiral de Rham complex on a Kummer surface. By chiral Poincar\'e duality in \cite{MS2} and the fact
that the elliptic genus of a K3 surface is  the Jacobi form $2 \phi_{0,1}(z;\tau)$ of weight $0$ and index $1$, all the graded dimensions of the cohomology of the chiral de Rham complex of a K3 surface can be calculated immediately, which has an application (see \cite{S2}) in Mathieu moonshine.
According to \cite{Kap} , on any Calabi-Yau manifold, the cohomology of the chiral de Rham complex
can be identified with the infinite-volume limit of the half-twisted sigma model defined by
E. Witten. This calculation may help to understand the half-twisted sigma model on K3 surfaces.

The paper is organized as follows. In section 2, we review some basic facts of complex geometry. In section 3, we  construct two holomorphic vector bundles $A_m(E,F)$ and $B_m(E,F)$.
In section 4, we compare the holomorphic structures of $A_m(E,F)$ and $B_m(E,F)$, and calculate the mean curvature of $A_m(E,F)$ when the mean curvatures of $E$, $F$ and $TX$ vanish.
In section 5, we give some results of the holomorphic sections of $A_m(E,F)$. Finally, in section 6, we calculate the global sections of the chiral de Rham complex on a K3 surface.
\section{The Chern connections and curvatures}
 For a complex manifold $X$, let $\Omega^{k,l}_X$ be the space of  smooth forms on $X$ of type $(k,l)$. Similarly, for a holomorphic vector bundle $E$ on $X$, let $\Omega_X^{k,l}(E)$ be the space of all smooth forms of type $(k,l)$ with values in $E$.  Let $h_E=(-,-)_E$ be a Hermitian metric on $E$ and
   $\nabla_E=\nabla_E^{1,0}+\bpartial$  be the associated Chern connection with $ \nabla_E^{1,0}: \Omega^{0,0}(E)\to \Omega^{1,0}(E)$.  Then its curvature is
 $$\Theta^E=[\bpartial, \nabla_E]=[\bpartial,\nabla_E^{1,0}]\in \Omega^{1,1}(\Hom(E,E)).$$
Locally, if $e=(e_1,\cdots, e_p)$ is a holomorphic frame   of $E$ over $U\subset X$, the connection one form $\theta^E$ of $\nabla_E$  is given by
$$\nabla_E e_i=\sum \theta^E_{ij}e_j.$$
If we define the matrix $H^E=(H_{ij}^E),$
$$H^E_{ij}=(e_i,e_j)_E,$$
then
$$\theta_{ij}^E=\partial H^E_{il} H_E^{lj}.$$
Here  $(H_E^{lj})$ is the inverse matrix of $H^E$.
The curvature form of $\nabla_E$ is
$$\Theta^E=\bpartial\theta^E\in \Omega^{1,1}(\Hom(E,E)).$$


\subsection{Covariant derivatives of the curvature.}
According to \cite{Ka}, the tensor fields $F_k=F_k^E$, $k\geq 2$ as higher covariant derivatives of the curvature can be defined by induction,
 \begin{eqnarray*}F_2&=&\Theta^E\in \Omega^{1,1}(\Hom(E,E))=\Omega^{0,1}(\Hom(T\otimes E,E)),\\
                 F_{k}&=&\nabla_E^{1,0} F_{k-1}\in \Omega^{0,1}(\Hom( T^{\otimes(k-1)}\otimes E,E)), \quad k\geq 3.
  \end{eqnarray*}
Since $\nabla_E$ is a Chern connection,
$$[\nabla^{1,0}_E,\nabla^{1,0}_E]=0.$$
So we have a lemma from \cite{Ka},
\begin{lemma}\label{lem:Fn}
   $$F_n\in \Omega^{0,1}(\Hom(\Sym^{n-1}T\otimes E,E)).$$
   \end{lemma}
   Locally, if $y=(y_1,\cdots, y_N)$ is a holomorphic coordinate system and $e=(e_1,\cdots, e_p)$ is a holomorphic frame of $E$ on $U\subset X$,
   $F_{k,i_1\cdots,i_n}^l$ is denoted by
   $$F_{k,i_1\cdots,i_n}^le_l=F_{n+1}(\frac{\partial}{\partial y_{i_1}},\cdots,\frac{\partial}{\partial y_{i_n}})e_k\in\Omega^{0,1}(E).$$

If $X$ is an Hermitian manifold with the Hermitian metric $h=(-,-)$, let $M=M^E$ be the mean curvature of $E$, i,e,
$$M=\Theta^E(\frac{\partial}{\partial y_{i}},\frac{\partial}{\partial \bar  y_{j}})H_{ij}\in \Hom(E,E).$$
Here  $H_{ij}=(dy_{i},d y_{i})$.
We define the  tensor field $M_n=M_n^E$  by
$$M_2=M,\quad M_n=\nabla^{1,0}M_{n-1},\quad n\geq 3.$$
Then
$$M_n=\sum F_n(\frac{\partial}{\partial y_{i}},\frac{\partial}{\partial \bar  y_{j}})H_{ij}\in \Hom(\Sym^{n-2}T\otimes E, E).$$
If $M=C_0 Id$ for a constant $C_0$, then $h$ is an Einstein-Hermitian metric.
We have the following lemma obviously,
\begin{lemma}\label{lem:mean}
If $h_E$ is an Einstein-Hermitian metric, then $M^E_{n}=0$ for $n\geq 3$.  If the mean curvature of $h_E$ is zero, then $M^E_n=0$ for $n\geq 2$. In particular, if the Ricci curvature of  $X$ is zero , then  $M_n^T=0$ for $n\geq 2$.
\end{lemma}

If  $X$ is K\"ahler, then for the  holomorphic tangent bundle $T$ of $X$,
$$\Theta^T\in \Omega^{0,1}(\Hom(\Sym^{2}T,T)),$$
we have
\begin{lemma}\label{lem:Rn}
$$F^T_n\in \Omega^{0,1}(\Hom(\Sym^{n}T,T)).$$
\end{lemma}

\subsection{The Chern connections of the vector bundle with different holomorphic structures.}
 Assume $E$ has another holomorphic structure given by  $\bpartial'$ and $\tilde\nabla$ is the Chern connection corresponding to  $\bpartial'$. Let $K=\bpartial-\bpartial'\in \Omega^{0,1}(\Hom(E,E))$. Let
    $K^*\in \Omega^{1,0}(\Hom(E,E))$ be the dual of $K$,
    i.e. $K^*$ is determined by
  $$ (K(\frac{\partial}{\partial \bar y_i}) a, b)=(a,K^*(\frac{\partial}{\partial  y_i}) b).$$
We have the following relation between the two connections,
  \begin{lemma}\label{lem:connection} $$\nabla=\tilde \nabla+K-K^*$$
    \end{lemma}
\begin{proof}
This is because that  the Chern connections $\nabla$ and $\tilde \nabla$ can be determined by
 $$\nabla=\nabla^{1,0}+\bpartial,  \quad \tilde \nabla=\tilde\nabla^{1,0}+\bpartial',$$
 and for any smooth sections $a,b $ of $E$,
$$d(a,b)=(\nabla a,b)+(a,\nabla b)=(\tilde\nabla a,b)+(a,\tilde\nabla b).$$
Now $$((\tilde \nabla+K-K^*)a,b)+(a,(\tilde \nabla+K-K^*)b)=(\tilde\nabla a,b)+(a,\tilde\nabla b)=d(a,b)$$ and
$$\tilde \nabla+K-K^*=\tilde\nabla^{1,0}+\bpartial'+K-K^*=\tilde\nabla^{1,0}-K^*+\bpartial$$
Since $K^*\in \Omega^{1,0}(\Hom(E,E)$, $\tilde\nabla^{1,0}-K^*$ takes values in $\Omega^{1,0}(E)$. So $\tilde \nabla+K-K^*$ is the Chern connection on $E$ corresponding to the holomorphic structure given by $\bpartial$. Thus by the uniqueness of the Chern connection, $\nabla=\tilde \nabla-K+K^*$.
\end{proof}

Let $E^{\vee}$ be the dual complex vector bundle of $E$, then $E^{\vee}$ has two holomorphic structures  induced from the two holomorphic structures on $E$.  The two holomorphic structures are given by  $\bpartial$ and $\bpartial'$, which are determined by
$$\bpartial (a^{\vee}(b))= (\bpartial a^{\vee})(b)+ a^{\vee}(\bpartial b)\,\text{ and } \bpartial (a^{\vee}(b))= (\bpartial'a^{\vee})(b)+ a^{\vee}(\bpartial' b)$$
for any section $a^{\vee}$ of $E^{\vee}$ and $b$ of $E$.
Let $\bpartial=\bpartial'+K^{\vee}$, then
$$\bpartial a^{\vee}(b)= ((\bpartial+K^{\vee})a^{\vee})(b)+ a^{\vee}((\bpartial+K) b).$$
Thus $K^{\vee}$ is determined by
$$(K^{\vee}a^{\vee})(b)=- a^{\vee}(K b).$$

\subsection{Weitzenb\"ock formulas}
 We list the following Weitzenb\"ock formulas for the compact Hermitian manifold here:
\begin{enumerate}
 \item If $\alpha$ is a smooth section of $E$,
\begin{eqnarray}\label{eqn:weitz1}
\Delta_{\bpartial}\alpha=\nabla^*\nabla \alpha-M\alpha.
\end{eqnarray}
\item If $X$ is K\"ahler, and $\alpha$ is a smooth section of $E=TX$, then the Ricci curvature
$$\Ric \alpha=M \alpha. $$
And
$$\Delta_{\bpartial}\alpha=\nabla^*\nabla\alpha-\Ric\alpha.$$

\end{enumerate}

\section{Vector bundles induced from jet schemes}

Given holomorphic bundles $E$ and $F$ on a complex manifold $X$, a family of holomorphic vector bundles $A_m(E,F)$, $ m\in\mathbb Z_{\geq 0}$ will be constructed in this section.

\subsection{The algebra $R_m$}
For $m\in\mathbb Z_{\geq 0}$, $ y=(y_1,\cdots,y_N)$, $e=(e_1,\cdots, e_p)$ and $f=(f_1,\cdots, f_q)$, let
$R_m=R_m(y,e,f)$ be the algebra $$\mathbb C[y_1^{(1)}, \cdots, y_i^{(j)},\cdots,y_N^{(m)}]\otimes \mathbb C[e_1^{(0)}, \cdots, e_i^{(j)},\cdots,e_p^{(m)}]\otimes \wedge_{\mathbb C}[f_1^{(0)}, \cdots, f_i^{(j)},\cdots,f_q^{(m)}].$$
Here  $y_i^{(j)}, e_i^{(j)},f_i^{(j)}$ are new variables and
$\wedge_{\mathbb C}[f_1^{(0)}, \cdots, f_i^{(j)},\cdots,,f_q^{(m)}]$ is the exterior algebra generated by $f_1^{(0)}, \cdots, f_i^{(j)},\cdots,f_q^{(m)}$ over $\mathbb C$. There is a natrual embeding $R_m(y,e,f)\hookrightarrow R_{m+1}(y,e,f)$ for all $m\geq 0$.
Let $$R_\infty(y,e,f)=\underrightarrow{\lim} R_m(y,e,f).$$

Let $$L=\sum_{i,j}  j \,y_i^{(j)}\frac{\partial }{\partial y_i^{(j)}}+\sum_{i,j}  j \,e_i^{(j)}\frac{\partial }{\partial e_i^{(j)}}+\sum_{i,j}  j \,f_i^{(j)}\frac{\partial }{\partial f_i^{(j)}},$$
$$L_e=\sum_{i,j}   e_i^{(j)}\frac{\partial }{\partial e_i^{(j)}},\quad L_f=\sum_{i,j}   \,f_i^{(j)}\frac{\partial }{\partial f_i^{(j)}}.$$
$L_e, L_f,L$ gives a $\mathbb Z^3_{\geq 0}$ grading of $R_m(y,e,f)$, i.e.
$$R_m(y,e,f)=\bigoplus_{j,k,l}R_m(y,e,f)[j,k,l]$$
with
$$L_e(a)=j a, \quad L_f(a)=k a,\quad L(a)=l a \quad \text{for }a\in R_m(y,e,f)[j,k,l].$$
It is easy to see that $R_m(y,e,f)[j,k,l]$ is a finite dimensional complex vector space.
Let  $$R_m(y,e,f)[j,k]=\bigoplus_{l}R_m(y,e,f)[j,k,l].$$

Let $\tilde D$  be the derivation on $R_\infty[y,e,f]$ given by
$$\tilde Dy_i^{(j)}=y_i^{(j+1)}, \quad \tilde De_i^{(j)}=e_i^{(j+1)}, \quad \tilde Df_i^{(j)}=f_i^{(j+1)}.$$
$\tilde D$ maps $R_m(y,e,f)[j,k,l]$ to $R_{m+1}(y,e,f)[j,k,l+1]$.
\subsection{Construction of the vector bundles $A_m(E,F)$}
Let $\cU$ be the set which consists of $(U,y,e,f)$ such that  $ y=(y_1,\cdots,y_N)$  is a coordinate system , $e=(e_1,\cdots, e_p)$ is a holomorphic frame of $E$ and $f=(f_1,\cdots, f_q)$ is a holomorphic frame of $F$ on $U$.
 Let $\cO(U)$ be the space of
holomorphic functions on $U$. For $m\in \mathbb Z_{\geq 0}\cup\{\infty\}$,
let $$\cA_{m}(U,y,e,f)=\cO(U)\otimes R_m(y,e,f)$$
be the algebra of holomorphic maps from $U$ to $R_m(y,e,f)$. For any $a\in \cA_{m}(U,y,e,f)$ , let $a(x) \in R_m(y,e,f) $ be the image of $a$ at $x\in U$.

Let
\begin{eqnarray}\label{eqn:D}
D=\sum_i\frac{\partial}{\partial y_i}\otimes y_i^{(1)}+1\otimes \tilde D.
 \end{eqnarray}
It is a derivation   on $\cA_{\infty}(U,y,e,f)$.

Now if $(U_\alpha,y_\alpha,e_\alpha,f_\alpha), (U_\beta,y_\beta,e_\beta,f_\beta)\in\cU$ with
$$y_{\alpha,i}=f_{i}(y_\beta),\quad e_{\alpha,i}=\sum g_{ij}e_{\beta,j}, \quad f_{\alpha,i}=\sum h_{ij}f_{\beta,j}\quad \text{on } U_\alpha\cap U_\beta,$$
there is  a unique  $\cO(U_\alpha\cap U_\beta)$ algebra isomorphism
    $$r_{\alpha\beta}: \cA_{\infty}(U_\alpha\cap U_\beta,y_\alpha,e_\alpha,f_\alpha)\to \cA_{\infty}(U_\alpha\cap U_\beta,y_\beta,e_\beta,f_\beta)$$
    with
    \begin{eqnarray}  r_{\alpha\beta}({y}_{\alpha,i}^{(l)})=D^lf(y_{\beta}),\quad r_{\alpha\beta}(\varphi)=\varphi,\quad \text{for  }\varphi\in\cO(U_\alpha\cap U_\beta),\nonumber\\
     r_{\alpha\beta}({e}_{\alpha,i}^{(l)})=D^l(\sum g_{ij}e_{\beta,j}^{(0)}),\quad r_{\alpha\beta}({f}_{\alpha,i}^{(l)})=D^l(\sum h_{ij}f_{\beta,j}^{(0)}).\nonumber
    \end{eqnarray}
It is easy to see
\begin{lemma}
    \label{lem:rD} For $a\in \cA_{\infty}(U_\alpha\cap U_\beta,y_\alpha,e_\alpha,f_\alpha)$, $$r_{\alpha\beta}(D a)=D\,r_{\alpha\beta}(a).$$
    \end{lemma}

For any $x\in U_\alpha\cap U_\beta$, we get an isomorphism of $\mathbb C$-algebras
$$r_{\alpha\beta}^x:R_\infty(y_\alpha,e_\alpha,f_\alpha)\to R_\infty(y_\beta,e_\beta,f_\beta)$$
with
\begin{eqnarray}\label{eqn:trans}
 r_{\alpha\beta}^x(y_{\alpha,i}^{(k)})=D^kf_i(y_\beta)(x),\quad
    r_{\alpha\beta}^x({e}_{\alpha,i}^{(k)})=D^k(\sum g_{ij}e_{\beta,j}^{(0)})(x),\\
     r_{\alpha\beta}^x({f}_{\alpha,i}^{(k)})=D^k(\sum h_{ij}f_{\beta,j}^{(0)})(x).\nonumber
\end{eqnarray}
$r_{\alpha\beta}^x$ satisfies
$$r_{\alpha\beta}^x\circ r_{\beta\alpha}^x=Id,\quad \text{for } x \in U_\alpha\cap U_\beta,$$
$$ r_{\beta\gamma}^x\circ r_{\alpha\beta}^x=r_{\alpha\gamma}^x, \quad \text{for } x\in U_\alpha\cap U_\beta\cap U_\gamma.$$
Thus  the open cover $\{U_\alpha\times R_{\infty}(y_\alpha,e_\alpha,f_\alpha) : (U_\alpha,y_\alpha,e_\alpha,f_\alpha)\in \cU\}$ and the transition functions $r_{\alpha\beta}^x$ define an algebra bundle $A_\infty(E,F)$ on $X$.

By (\ref{eqn:trans}), the transition functions $r_{\alpha\beta}^x$  map $R_m(y_\alpha,e_\alpha,f_\alpha)[j,k,l]$ to $R_m(y_\beta,e_\beta,f_\beta)[j,k,l]$.
Let $A_m(E,F)[j,k,l]$, $A_m(E,F)[j,k]$ and $A_m(E,F)$ be the subbundles of $A_\infty[E,F]$ with fibres $R_m[j,k,l]$,  $R_m[j,k]$ and $R_m$, respectively.  We have  $$A_m(E,F)=\bigoplus_{j,k,l} A_m(E,F)[j,k,l],$$
 $$A_m(E,F)[j,k]=\bigoplus_{l} A_m(E,F)[j,k,l].$$
 $A_m(E,F)[j,k,l]$ are finite dimensional
holomorphic vector bundles.
Let $O_m$ be the algebra bundle $A_m(E,F)[0,0]=A_m(0,0)$ and $O_m[l]=A_m(E,F)[0,0,l]$.

Let $\cA_m(E,F)$, $\cA_m(E,F)[j,k]$ and $\cO_m$ be the sheaf of holomorphic sections of $A_m(E,F)$, $A_m(E,F)[j,k]$ and $O_m$, respectively. $\cA_m(E,F)$ are algebras of $\cO_m$ and $\cA_m(E,F)[j,k]$ are modules of $\cO_m$.

\begin{remark} If $X$ is a complex projective manifold, then $\cA_m(E,F)[j,k]$ can be induced from the $m$th jet scheme of $X$.  For each integer $m\geq 0$,
the sheaf $\cO_m$ is the push forward of the structure sheaf of $J_m(X)$ through $\pi_m$.
$J_m(E)$ and $J_m(F)$ are vector bundles over $J_m(X)$. $\cA_m(E,F)[j,k]$ is the push forward of the sheaf of sections of $\Sym^jJ_m(E)\otimes\wedge^k J_m(F)$ through $\pi_m$.
\end{remark}

By Lemma \ref{lem:rD}, $r_{\alpha\beta}$ commutes with the derivation $D$. So $D$ is a derivation of the sheaf of algebras $\cA_\infty(E,F)$.  According to  (\ref{eqn:D}),
$D$ maps $\cA_m(E,F)[j,k,l]$ to $\cA_{m+1} (E,F)[j,k,l+1]$. Thus $D$  is a derivation of the $\cO_\infty$ module
$\cA_\infty(E,F)[j,k]$.

$D$ can be extended to a derivation of $\Omega_U^{0,*}(A_\infty(E,F))$ by assuming $D\bar dy_i=0$ and $Df=\sum_i\frac {\partial f}{\partial y_i}y_i^{(1)}$, for any smooth function $f$ on $U$.

\subsection{Vector bundles $B_m(E,F)$.}
Let $B_m(E,F)$ be the algebra bundle constructed from the open cover $\{U_\alpha\times R_{\infty}(y_\alpha,e_\alpha,f_\alpha) : (U_\alpha,y_\alpha,e_\alpha,f_\alpha)\in \cU\}$ and the transition functions
$$\tilde r_{\alpha\beta}^x:R_\infty(y_\alpha,e_\alpha,f_\alpha)\to R_\infty(y_\beta,e_\beta,f_\beta).$$
For $x\in U_\alpha\cap U_\beta$, $\tilde r_{\alpha\beta}^x$ is the isomorphism of algebras given by
\begin{eqnarray}
 \tilde r_{\alpha\beta}^x(y_{\alpha,i}^{(k)})=\sum \frac{\partial f_i(y_\beta)}{\partial y_{\beta,j}} (x)y_{\beta,j}^{(k)},
    r_{\alpha\beta}({e}_{\alpha,i}^{(k)})=\sum g_{ij}(x)e_{\beta,j}^{(k)}(x),\, r_{\alpha\beta}({f}_{\alpha,i}^{(k)})=\sum h_{ij}(x)f_{\beta,j}^{(k)}.\nonumber
\end{eqnarray}
$\tilde r_{\alpha\beta}^x$ maps $R_m(y_\alpha,e_\alpha,f_\alpha)[j,k,l]$to $ R_m(y_\beta,e_\beta,f_\beta)[j,k,l]$.
Let $B_m(E,F)[j,k,l]$, $B_m(E,F)[j,k]$ and $B_m(E,F)$ be the subbundles of $B_\infty[E,F]$ with fibres $R_m[j,k,l]$,  $R_m[j,k]$ and $R_m$ respectively.
 Then  $B_m(E,F)[j,k,l]$ are finite dimensional
holomorphic vector bundles. We have
\begin{eqnarray}\label{eqn:Bm}
B_m(E,F)&=&\bigoplus_{j,k,l} B_m(E,F)[j,k,l]\\
&=&\Sym^*(\oplus_{1\leq l\leq m} T^*)\otimes \Sym^*(\oplus_{0\leq l\leq m} E)\otimes\wedge^*(\oplus_{0\leq l\leq m} F),\nonumber
\end{eqnarray}
\begin{eqnarray}\label{eqn:Bmjk}
 B_m(E,F)[j,k]&=&\bigoplus_{l} B_m(E,F)[j,k,l]\\
 &=&\Sym^*(\oplus_{1\leq l\leq m} T^*)\otimes \Sym^j(\oplus_{0\leq l\leq m} E)\otimes\wedge^k(\oplus_{0\leq l\leq m} F).\nonumber
 \end{eqnarray}

It is easy to see that $\tilde D$ commutes with $r^x_{\alpha\beta}$. So $\tilde D$ is a derivation of the algebra bundle $B_m(E,F)$.

 Let $O_{m,1}=\oplus_{k>0}O_m[k]$ and $O_{m,k}=(O_{m,1})^k$, the sub bundle of $O_m$ generated by product of $k$ elements of $O_m$ with positive weights. Thus
$$\cdots O_{m,k}\subset\cdots\subset O_{m,2}\subset O_{m,1}\subset O_{m,0}=\cO_m.$$
$$\bigoplus_{k\geq 0}O_{m,k}\slash O_{m,k+1}\cong \bigoplus_{k\geq 0}\Sym^k(\oplus_{1\leq l\leq m} T^*_l).$$
In general we have
$$\bigoplus_{k\geq 0}A(E,F)[j,k]O_{m,l}\slash A(E,F)[j,k]O_{m,l+1}=B_m(E,F)[j,k].$$

\section{Holomorphic structures}
Let $h_{T^*}=(-,-)_{T^*}$ be an Hermitian metric of the cotangent bundle $T^*$ of $X$, and $h_E=(-,-)_E$ and $h_F=(-,-)_F$ be Hermitian metrics of $E$ and $F$ respectively.
Let $\nabla_{T^*}$, $\nabla_E$ and $\nabla_F$ be the Chern connections of $T^*$, $E$ and $F$ respectively.
For $(U,y,e,f)\in \cU$,
$$\nabla_T dy_i=\sum \theta^{T^*}_{ij} dy_j,\nabla_E e_i=\sum \theta^E_{ij}e_j, \nabla_F f_i=\sum \theta^F_{ij}f_j.$$
 Let $\Theta^{T^*}$, $\Theta^E$ and $\Theta^F$ be their curvatures.

\subsection{Chern connection on $B_m(E,F)$}  By  (\ref{eqn:Bm}), there is a canonical Hermitian metric $h=(-,-)$ on $B_m(E,F)$ induced from $h_{T^*}$, $h_E$ and $h_F$.
Let $\tilde \nabla=\tilde \nabla^{1,0}+\bpartial'$ be the Chern connection associated with $h$. For $(U,y,e,f)\in \cU$, $\tilde \nabla$ satisfies
\begin{eqnarray}\label{eqn:conB}
\tilde \nabla  y_i^{(l)}=\sum \theta^{T^*}_{ij}y_j^{(l)},\quad \tilde \nabla e_i^{(l)}=\sum \theta^E_{ij}e_j^{(l)},\quad \tilde \nabla f_i^{(l)}=\sum \theta^F_{ij}f_j^{(l)},\\
\tilde\nabla (ab)=(\tilde\nabla a)b+a(\tilde\nabla b),\quad \text{for } a, b\in\Omega_U^{0,0}(B_m(E,F)).\nonumber
\end{eqnarray}
So its curvature $\tilde \Theta$ satisfies
\begin{eqnarray}\label{eqn:curB}
\tilde\Theta  y_i^{(l)}=\sum \Theta^{T^*}_{ij}y_j^{(l)},\quad \tilde\Theta e_i^{(l)}=\sum \Theta^E_{ij}e_j^{(l)},\quad \tilde\Theta f_i^{(l)}=\sum \Theta^F_{ij}f_j^{(l)},\\
\tilde \Theta (ab)=(\tilde\Theta a)b+a(\tilde\Theta b),\quad \text{for } a, b\in\Omega_U^{0,0}(B_m(E,F)).\nonumber
\end{eqnarray}
\subsection{A canonical isomorphism.}
$A_m(E,F)$ and $B_m(E,F)$ are different holomorphic vector bundles. But as complex vector bundles, they are isomorphic. Here we construct a canonical isomorphism from $A_m(E,F)$ to $B_m(E,F)$.

For $(U,y,e,f)\in \cU$,
 let
  $$H_{ij}=(dy_i,dy_j)_{T^*},\quad H_{ij}^E=(e_i,e_j)_E,\quad H_{ij}^F=(f_i,f_j)_F ,$$
  and let $(H^{ij})$,  $(H_E^{ij})$ and  $(H_F^{ij})$ be the inverse matrices of $(H_{ij})$, $H_{ij}^E)$ and $(H_{ij}^F)$ respectively.
 For $k\geq 0$,
 let
\begin{eqnarray}\label{eqn:YEF}Y_l^{(k+1)}=\sum H_{lj}D^{k}(H^{ji}y_i^{(1)}),\quad
E_l^{(k)}=\sum H_{lj}^ED^{k}(H^{ji}_EE_i^{(0)}),\quad
F_l^{(k)}=\sum H_{lj}^FD^{k}(H^{ji}_FF_i^{(0)}).
\end{eqnarray}

\begin{prop}\label{prop:isoAB}
There is an isomorphism  from  $\Phi_m : A_m(E,F)\to B_m(E,F)$ of smooth complex vector bundles,  which locally maps ${Y_l}^{(k)}$, ${E_l}^{(k)}$ and  ${F_l}^{(k)}$ to
  ${y_l}^{(k)}$, ${e_l}^{(k)}$ and  ${f_l}^{(k)}$ respectively. This isomorphism maps $A_m(E,F)[j,k,l]$ to $B_m(E,F)[j,k,l]$ and preserves the algebra structures.
\end{prop}
\begin{proof}For any point $x\in X$, let $(U_\alpha,y_\alpha,e_\alpha,f_\alpha)\in\cU$ with $x\in U_\alpha$. By the definition of $Y_l^{(k+1)}$, $E_l^{(k)}$ and $F_l^{(k)}$ in (\ref{eqn:YEF}),
\begin{eqnarray}Y_{\alpha,l}^{(k+1)}\equiv y_{\alpha,l}^{(k+1)} \mod O_{m,2};\nonumber\\
E_{\alpha,l}^{(k)}\equiv e_{\alpha,l}^{(k)},\quad
F_{\alpha,l}^{(k)}\equiv f_{\alpha,l}^{(k)}\quad \mod A_m(E,F)O_{m,1}. \nonumber
\end{eqnarray}
So $R_m(Y_\alpha, E_\alpha, F_\alpha)$ are the fibres of $A_m(E,F)$ at $x$.

Let  $\Phi_{m,x} : A_m(E,F)|_x\to B_m(E,F)|_x$ be the isomorphism of $\mathbb C$-algebras given locally by
$$\Phi_{m,x}(Y_{\alpha,l}^{(k+1)})=y_{\alpha,l}^{(k+1)},\quad \Phi_{m,x}(E_{\alpha,l}^{(k)})=e_{\alpha,l}^{(k)},\quad \Phi_{m,x}(F_{\alpha,l}^{(k)})=f_{\alpha,l}^{(k)}.$$
$Y_{\alpha,l}^{(k+1)}$,$E_{\alpha,l}^{(k)}$ and  $F_{\alpha,l}^{(k)}$ have the same grades as $y_{\alpha,l}^{(k+1)}$,$e_{\alpha,l}^{(k)}$ and  $f_{\alpha,l}^{(k)}$ respectively. So
$\Phi_{m,x}$ preserves the $\mathbb Z^3_{\geq 0}$ grading of each fibre.

 If $(U_\beta,y_\beta,e_\beta,f_\beta)\in \cU$ with $x\in U_\beta$,
$y_{\alpha,i}=\sum f_{i}(y_\beta)$, $e_{\alpha,i}=\sum g_{ij}e_{\beta,j}$, $ f_{\alpha,i}=\sum h_{ij}f_{\beta,j}$ on $U_\alpha\cap U_\beta.$
$$H_{\beta, ij}=(dy_{\beta,i},dy_{\beta,j})=\sum \frac{\partial y_{\beta,i}}{\partial y_{\alpha,k}}H_{\alpha,kl}\overline{\frac{\partial y_{\beta,j}}{\partial y_{\alpha,l}}}.$$
So
\begin{eqnarray}\label{eqn:Y}\gamma_{\alpha\beta}(Y_{\alpha,l}^{(k+1)})&=&\sum H_{\alpha,lj}D^{k}(H_\alpha^{ji}\gamma_{\alpha\beta}(y_{\alpha,i}^{(1)}))\\
&=&\sum (\frac{\partial y_{\alpha,l}}{\partial y_{\beta,s}}H_{\beta,st}\overline{\frac{\partial y_{\alpha,j}}{\partial y_{\beta,t}}})D^k
(\overline{\frac{\partial y_{\beta,u}}{\partial y_{\alpha,j}}}H_{\beta}^{uv}\frac{\partial y_{\beta,v}}{\partial y_{\alpha,i}}\gamma_{\alpha\beta}(y_{\alpha,i}^{(1)}))\nonumber \\
&=&\sum (\frac{\partial y_{\alpha,l}}{\partial y_{\beta,s}}H_{\beta,st}\overline{\frac{\partial y_{\alpha,j}}{\partial y_{\beta,t}}}\, \overline{\frac{\partial y_{\beta,u}}{\partial y_{\alpha,j}}})D^k
(H_{\beta}^{uv}\frac{\partial y_{\beta,v}}{\partial y_{\alpha,i}}\gamma_{\alpha\beta}(y_{\alpha,i}^{(1)}))\nonumber\\
&=&\sum \frac{\partial y_{\alpha,l}}{\partial y_{\beta,s}}H_{\beta,su}D^k(H_{\beta}^{uv}y_{\beta,v}^{(1)})\nonumber\\
&=&\sum \frac{\partial y_{\alpha,l}}{\partial y_{\beta,s}}Y_{\beta,s}^{(k+1)}.\nonumber
\end{eqnarray}
Similarly,
\begin{eqnarray}
\gamma_{\alpha\beta}(E_{\alpha,l}^{(k)})=\sum g_{ls}E_{\beta,s}^{(k)},\quad \gamma_{\alpha\beta}(F_{\alpha,l}^{(k)})=\sum h_{ls}F_{\beta,s}^{(k)}.
\end{eqnarray}
Thus $\Phi_{m,x}\circ \gamma_{\alpha\beta}^x=\tilde\gamma_{\alpha\beta}^x\circ \Phi_{m,x}$. $\Phi_{m,x}$ is independent of the choice of $(U_\alpha,y_\alpha,e_\alpha,f_\alpha)\in\cU$.

$\Phi_{m,x}$ smoothly depends on $x$, so it gives an isomorphism $\Phi_m : A_m(E,F)\to B_m(E,F)$ of smooth complex vector bundles. $\Phi_{m,x}$ preserves the grading, so
$\Phi_m$ maps  $ A_m(E,F)[j,k,l]$ to $B_m(E,F)[j,k,l]$. $\Phi_{m,x}$ is an isomorphism of $\mathbb C$-algebra, so $\Phi_m$ preserves the algebra structure.
\end{proof}

\subsection{Holomorphic structures on $A_m(E,F)$}
Through the isomorphism $\Phi_m$,  $B_m(E,F)$ can be regarded as the same smooth complex vector bundle as  $A_m(E,F)$ with a different holomorphic structure.
Now the underlying complex vector bundle of $A_m(E,F)$ has a Hermitian metric $(-,-)_m$ ( from $B_m(E,F)$) and two holomorphic structures. One is from $A_m(E,F)$, which is determined by $\bpartial$ with
\begin{eqnarray}\label{eqn:bp}
\bpartial y_{i}^{(l)}=0,\quad \bpartial e_{i}^{(l)}=0,\quad \bpartial f_{i}^{(l)}=0, \\
  \bpartial ab=\bpartial a b+a\bpartial b,\quad a,b\in \Omega^{0,0}(A_m(E,F)).\nonumber
  \end{eqnarray}
The other  is induced from $B_m(E,F)$, which is determined by $\bpartial'$ with
\begin{eqnarray}\label{eqn:bpp}
\bpartial' Y_{i}^{(l)}=0,\quad \bpartial' E_{i}^{(l)}=0,\quad \bpartial' F_{i}^{(l)}=0,\\
\bpartial' ab=\bpartial' a b+a\bpartial' b,\quad a,b\in \Omega^{0,0}(A_m(E,F)).\nonumber
\end{eqnarray}

Under the fixed Hermitian metric $(-,-)_m$ on $A_m(E,F)$, let $\nabla=\nabla^{1,0}+\bpartial$ and  $\tilde \nabla=\tilde \nabla^{1,0}+ \bpartial '$ be the Chern connections associated with $\bpartial$ and $\bpartial'$ respectively. By (\ref{eqn:conB}) and (\ref{eqn:curB}) and Proposition \ref{prop:isoAB}, for $(U,y,e,f)\in \cU$, $\tilde \nabla$ satisfies
\begin{eqnarray}\label{eqn:conB1}
\tilde \nabla  Y_i^{(l)}=\sum \theta^{T^*}_{ij}Y_j^{(l)},\quad \tilde \nabla E_i^{(l)}=\sum \theta^E_{ij}E_j^{(l)},\quad \tilde \nabla F_i^{(l)}=\sum \theta^F_{ij}F_j^{(l)},\\
\tilde\nabla (ab)=(\tilde\nabla a)b+a(\tilde\nabla b),\quad \text{for } a, b\in\Omega_U^{0,0}(A_m(E,F)).\nonumber
\end{eqnarray}
Its curvature $\tilde \Theta$ satisfies
\begin{eqnarray}\label{eqn:curB1}
\tilde\Theta  Y_i^{(l)}=\sum \Theta^{T^*}_{ij}Y_j^{(l)},\quad \tilde\Theta E_i^{(l)}=\sum \Theta^E_{ij}E_j^{(l)},\quad \tilde\Theta F_i^{(l)}=\sum \Theta^F_{ij}F_j^{(l)},\\
\tilde \Theta (ab)=(\tilde\Theta a)b+a(\tilde\Theta b),\quad \text{for } a, b\in\Omega_U^{0,0}(A_m(E,F)).\nonumber
\end{eqnarray}

There is also a derivation $\tilde D: A_{\infty}(E,F)\to  A_{\infty}(E,F)$ induced from $\tilde D$ on $B_m(E,F)$ through $\Phi_\infty$. By Proposition \ref{prop:isoAB}, for $(U,y,e,f)\in \cU$,
$\tilde D$ satisfies
$$\tilde DY_i^{(j)}=Y_i^{(j+1)}, \quad \tilde DE_i^{(j)}=E_i^{(j+1)}, \quad \tilde DF_i^{(j)}=F_i^{(j+1)}.$$
\begin{lemma}\label{lem:DD}
$$D=\tilde D + \sum Y_t^{(1)}\tilde\nabla_{\frac{\partial}{\partial y_t}}.$$
\end{lemma}
\begin{proof}
$D$, $\tilde D$ and $\sum_t Y_t^{(1)}\tilde\nabla_{\frac{\partial}{\partial y_t}}$ are globally defined operators.  To prove the equation, we only need to prove it locally.
For $(U,y,e,f)\in \cU$,
\begin{eqnarray*}DY_l^{(k+1)}&=&\sum D(H_{lj}D^{k}(H^{ji}y_i^{(1)}))\\
                             &=&\sum (D H_{lj}) D^{k}(H^{ji}y_i^{(1)})+\sum H_{lj}D^{k+1}(H^{ji}y_i^{(1)})\\
                              &=&\sum Y_t^{(1)}\frac{\partial H_{lj}}{\partial y_t}H^{js}H_{sm}(D^{k}(H^{mi}y_i^{(1)}))+Y_l^{(k+2)}\\
                             &=&Y_l^{(k+2)}+\sum Y_t^{(1)} \theta_{ls}^{T^*}(\frac{\partial}{\partial y_t})Y_s^{(k+1)}.\\
                             &=&\tilde D Y_l^{k+1}+\sum Y_t^{(1)}\tilde\nabla_{\frac{\partial}{\partial y_t}}Y_l^{k+1}
                                                          \end{eqnarray*}
Similarly,
\begin{eqnarray*} DE_l^{(k)}&=&\tilde D E_l^{k}+\sum Y_t^{(1)}\tilde\nabla_{\frac{\partial}{\partial y_t}}E_l^{k}\\
 DF_l^{(k)}&=&\tilde D F_l^{k}+\sum Y_t^{(1)}\tilde\nabla_{\frac{\partial}{\partial y_t}}F_l^{k}\\
\end{eqnarray*}

For a smooth function $f$ on $U$,
$$\tilde Df+\sum Y_t^{(1)}\tilde\nabla_{\frac{\partial}{\partial y_t}}f=0+\sum Y_t^{(1)}\frac{\partial f}{\partial y_t}=Df.$$
Since both $D$ and $\tilde D +\sum  Y_t^{(1)}\tilde\nabla_{\frac{\partial}{\partial y_t}}$ are derivations on $\Omega^{0,0}_U(A_\infty(E,F))$ and $Y_l^{(k+1)}$, $E_l^{(k)}$,  $F_l^{(k)}$
and smooth functions on $U$ generate  $\Omega^{0,0}_U(A_\infty(E,F))$,  $D$  is  equal to $\tilde D + \sum Y_t^{(1)}\tilde\nabla_{\frac{\partial}{\partial y_t}}$.
\end{proof}

We  can find the relation between $\bpartial$ and $\bpartial'$ by the above lemma and the following fact.
\begin{lemma}\label{lem:comD}
 $$[D,\bpartial]=0, \quad [\tilde D,\bpartial']=0$$
\end{lemma}
\begin{proof}$[D,\bpartial]$ and $[\tilde D,\bpartial']$ are  derivations on the sheaf of smooth sections of $A_\infty(E,F)$. Locally,  for  $(U,y,e,f)\in \cU$ and a smooth function $f$ on $U$,
$$[D,\bpartial]y_l^{(k+1)}=[D,\bpartial]e_l^{(k)}=[D,\bpartial]f_l^{(k)}=[D,\bpartial]f=0.$$
$$[\tilde D,\bpartial']Y_l^{(k+1)}=[\tilde D,\bpartial']E_l^{(k)}=[\tilde D,\bpartial']F_l^{(k)}=[\tilde D,\bpartial']f=0.$$
\end{proof}

Let $K=\bpartial-\bpartial'$.  For any smooth function $f$ on $X$,  $Kf=\bpartial f-\bpartial'f=0$.
By (\ref{eqn:bp}) and (\ref{eqn:bpp}),
$$K(ab)=(Ka)b+a(Kb),\quad \text{for } a, b\in \Omega^{0,0}(A_m(E,F)).$$
So $K$ is determined if we know $KY_l^{(k+1)}$,  $KE_l^{(k)}$ and $KF_l^{(k)}$. Using the following lemma, we can calculate them out.

\begin{lemma}\label{lem:KD}
$$[K,\tilde D]= \sum Y_t^{(1)}[\tilde\nabla_{\frac{\partial}{\partial y_t}},\bpartial]=\sum  -Y_t^{(1)}\tilde \Theta(\frac{\partial}{\partial y_t})+\sum Y_t^{(1)}[\tilde\nabla_{\frac{\partial}{\partial y_t}},K].$$
\end{lemma}
\begin{proof}$\bpartial {Y_l}^{(1)}=\bpartial y_l^{(1)}=0.$
\begin{eqnarray*}
[K,\tilde D]&=&[\bpartial-\bpartial',\tilde D]\\
            &=&[\bpartial, \tilde D] \quad\quad\quad\quad\quad\quad\quad\quad\quad\text{(by Lemma \ref{lem:comD})}\\
            &=&[\bpartial, D-\sum Y_t^{(1)}\tilde\nabla_{\frac{\partial}{\partial y_t}}]\quad \quad\quad\quad\text{(by Lemma \ref{lem:DD})}\\
            &=&\sum Y_t^{(1)}[\tilde\nabla_{\frac{\partial}{\partial y_t}},\bpartial]\quad\quad\quad\quad\quad\quad\text{(by Lemma \ref{lem:comD})}\\
            &=&\sum Y_t^{(1)}[\tilde\nabla_{\frac{\partial}{\partial y_t}},\bpartial'+K]\\
            &=&\sum -Y_t^{(1)}\tilde \Theta(\frac{\partial}{\partial y_t})+\sum Y_t^{(1)}[\tilde\nabla_{\frac{\partial}{\partial y_t}},K].
\end{eqnarray*}
\end{proof}
Let
$C_{j_1,\cdots,j_a}=-\frac 1{(a-1)!}\frac {(j_1+\cdots+j_a)!}{j_1!j_2!\cdots j_a!}$, if $j_1\geq 0$ and $j_l\geq 1$ for $l\geq 2$, otherwise $C_{j_1,\cdots,j_a}=0$.
\begin{lemma}\label{lem:Dbar}
For $k\geq 0$, If $P_l^{(k)}$ is  $Y_l^{(k+1)}$,  $E_l^{(k)}$ or $F_l^{(k)}$, and $F$ is $F^{T^*}$, $F^{E}$ or $F^{F}$ respectively, then
$$K P_l^{(k)}=\sum_{a=2}^k\sum_{j_1+\cdots+j_a=k}\sum_{i_1,\cdots,i_a}C_{j_1,\cdots,j_a}F^l_{i_1,\cdots, i_a}P_{i_1}^{(j_1)}Y_{i_2}^{(j_1)}\cdots Y_{i_a}^{(j_a)}.$$
\end{lemma}
\begin{proof} Here we give the proof for $P_l^{(k)}=Y_l^{(k+1)}$. The proofs are similar for     $P_l^{(k)}=E_l^{(k)}$ and $P_l^{(k)}=F_l^{(k)}$.

We prove it by induction on $k$. When $n=0$, $ {Y_l}^{(1)}=y_l^{(1)}$. So $\bpartial Y_l^{(1)}=\bpartial y_l^{(1)}=0$. Since  $\bpartial'  Y_l^{(1)}=0,$
$K Y_l^{(1)}=0.$ the lemma is true.
Assume the lemma is true for $n=k$. If $n=k+1$, by Lemma \ref{lem:KD},
\begin{eqnarray*}
KY_l^{(k+2)}&=&K\tilde D Y_l^{(k+1)}\\
&=&\tilde D K Y_l^{(k+1)}-\sum Y_t^{(1)}\tilde \Theta(\frac{\partial}{\partial y_t})Y_l^{(k+1)}+\sum Y_t^{(1)}[\tilde\nabla_{\frac{\partial}{\partial y_t}},K]Y_l^{(k+1)} \quad(\text{By Lemma\ref{lem:KD}})\\
&=&\tilde D(\sum_{a=2}^k\sum_{j_1+\cdots+j_a=k}C_{j_1,\cdots,j_a}\sum_{i_1,\cdots,i_a}F^l_{i_1,\cdots, i_a}Y_{i_1}^{(j_1+1)}Y_{i_2}^{(j_1)}\cdots Y_{i_a}^{(j_a)})\\
& &-\sum Y_t^{(1)}\Theta^{T^*}_{ls}(\frac{\partial}{\partial y_t})Y_{s}^{(k+1)}\\
&&+\sum Y_t^{(1)}\sum_{a=2}^k\sum_{j_1+\cdots+j_a=k}\sum_{i_1,\cdots,i_a}C_{j_1,\cdots,j_a}(\nabla_{\frac{\partial}{\partial y_t}}^{T^*}F_a)^l_{i_1,\cdots, i_a}Y_{i_1}^{(j_1+1)}Y_{i_2}^{(j_1)}\cdots Y_{i_a}^{(j_a)}\\
&=&\sum_{a=2}^k\sum_{j_1+\cdots+j_a=k+1}\sum_{i_1,\cdots,i_a}(\sum_{l=1}^a C_{j_1,\cdots,j_l-1,\cdots,j_a})F^l_{i_1,\cdots, i_a}Y_{i_1}^{(j_1+1)}Y_{i_2}^{(j_1)}\cdots Y_{i_a}^{(j_a)}\\
& &-F_{i_1,i_2}^2Y_{i_1}^{(k+1)}Y_{i_2}^{(1)}\\
& &+\sum_{a=2}^k\sum_{j_1+\cdots+j_a=k}\sum_{i_1,\cdots,i_a}\frac{a}{k+1}C_{j_1,\cdots,j_a,1}F^l_{i_1,\cdots, i_a,i_{a+1}}Y_{i_1}^{(j_1+1)}Y_{i_2}^{(j_1)}\cdots Y_{i_a}^{(j_a)}Y_{i_{a+1}}^{(1)}\\
&=&\sum_{a=2}^{k+1}\sum_{j_1+\cdots+j_a=k+1}\sum_{i_1,\cdots,i_a}C_{j_1,\cdots,j_a}F^l_{i_1,\cdots, i_a}Y_{i_1}^{(j_1+1)}Y_{i_2}^{(j_1)}\cdots Y_{i_a}^{(j_a)}.
\end{eqnarray*}
So the lemma is true  for $n=k+1$. Thus the lemma is true for any $k\geq 0$.
\end{proof}
By Lemma \ref{lem:Dbar}, we immediately have
\begin{lemma}\label{lem:KP}
For $k\geq 0$, If $P_l^{(k)}$ is  $Y_l^{(k+1)}$,  $E_l^{(k)}$ or $F_l^{(k)}$, and $F$ is $F^{T^*}$, $F^{E}$ or $F^{F}$ respectively, then
$$[\tilde\nabla^{1,0}, K]P_l^{(k)}=\sum_{a=2}^k\sum_{j_1+\cdots+j_a=k}\sum_{i_1,\cdots,i_{a+1}}C_{j_1,\cdots,j_a}dy_{i_{a+1}}\wedge F^l_{i_1,\cdots, i_a,i_{a+1}}P_{i_1}^{(j_1)}Y_{i_2}^{(j_1)}\cdots Y_{i_a}^{(j_a)}.$$
\end{lemma}

\subsection{The mean curvature of $A_m(E,F)$}

By Lemma \ref{lem:connection}, the connections on $A_m(E,F)$ satisfy
$$\nabla=\tilde \nabla+K-K^*.$$
So their curvatures satisfy
\begin{eqnarray}\label{eqn:theta}
\Theta=[\bpartial, \nabla]=[\bpartial'+K,\tilde\nabla^{1,0}-K^*]=\tilde \Theta+[K,\tilde\nabla^{1,0}]-[\bpartial', K^*]-[K,K^*].
\end{eqnarray}
Now $\tilde\nabla^{1,0}$ and $K$ are derivations on the sheaf of smooth sections of $A_m(E,F)$, so $[\tilde\nabla^{1,0}, K]$ is a derivation on the sheaf of smooth sections of $A_m(E,F)$.
\begin{lemma}\label{lem:connK}
If   $h_{T^*}$, $h_E$ and $h_F$ are  Einstein-Hermitian metrics then
$$\sum [\tilde\nabla^{1,0},K] (\frac{\partial}{\partial y_i}, \frac{\partial}{\partial \bar y_j})H_{ij}=0.$$
\end{lemma}
\begin{proof}
 $\tilde\nabla^{1,0}$ and $K$ are derivations on the sheaf of smooth sections of $A_m(E,F)$, so $[\tilde\nabla^{1,0}, K]$ and $\sum [\tilde\nabla^{1,0},K] (\frac{\partial}{\partial y_i}, \frac{\partial}{\partial \bar y_j})H_{ij}$ are derivations on the sheaf of smooth sections of $A_m(E,F)$.
 From Lemma \ref{lem:KP}, if $P_l^{(k)}$ is  $Y_l^{(k+1)}$,  $E_l^{(k)}$ or $F_l^{(k)}$, and $F$ is $F^{T^*}$, $F^{E}$ or $F^{F}$ respectively,
$$[\tilde\nabla^{1,0},K] P_l^{(k)}(\frac{\partial}{\partial y_i}, \frac{\partial}{\partial \bar y_j})H_{ij}=\sum_{a=2}^k\sum_{j_1+\cdots+j_a=k}\sum_{i_1,\cdots,i_a}C_{j_1,\cdots,j_a} F^l_{i_1,\cdots, i_a,i}(\frac{\partial}{\partial \bar y_j})H_{ij}P_{i_1}^{(j_1)}Y_{i_2}^{(j_1)}\cdots Y_{i_a}^{(j_a)}.$$
By Lemma \ref{lem:mean}, if $h_{T^*}$, $h_E$ and $h_F$  are Einstein-Hermitian metrics, then for $n\geq 3$,  $M^{T^*}_n$,  $M^{E}_n$ and $M^{F}_n$ all vanish. So
$$\sum F^l_{i_1,\cdots, i_a,i}(\frac{\partial}{\partial \bar y_j})H_{ij}=0.$$
Thus $\sum [\tilde\nabla^{1,0},K] P_l^{(k)}(\frac{\partial}{\partial y_i}, \frac{\partial}{\partial \bar y_j})H_{ij}=0$. For any smooth function $f$ on $X$,
$$[\tilde\nabla^{1,0},K]f=\tilde\nabla^{1,0} Kf+K\tilde\nabla^{1,0}f=K \partial f=0.$$ Since $\sum [\tilde\nabla^{1,0},K] (\frac{\partial}{\partial y_i}, \frac{\partial}{\partial \bar y_j})H_{ij}$
is a derivation on the sheaf of smooth sections of $A_m(E,F)$, it is zero.
\end{proof}

 We have the following theorem for the mean curvature of the connection $\nabla$ on $A_m(E,F)$.
\begin{thm}\label{thm:curv}
If   $h_{T^*}$, $h_E$ and $h_F$ are  Einstein-Hermitian metrics then the mean curvature of $\nabla$ is
$$\sum \Theta(\frac{\partial}{\partial y_i}, \frac{\partial}{\partial \bar y_j})H_{ij}=\sum \tilde\Theta(\frac{\partial}{\partial y_i}, \frac{\partial}{\partial \bar y_j})H_{ij}+\sum K(\frac{\partial}{\partial \bar y_j})K^*(\frac{\partial}{\partial y_i})H_{ij}-\sum K^*(\frac{\partial}{\partial y_i})K(\frac{\partial}{\partial \bar y_j})H_{ij}.$$
In particular, if the mean curvatures of  $h_{T^*}$, $h_E$ and $h_F$ are zero,
$$\sum \Theta(\frac{\partial}{\partial y_i}, \frac{\partial}{\partial \bar y_j})H_{ij}=\sum K(\frac{\partial}{\partial \bar y_j})K^*(\frac{\partial}{\partial y_i})H_{ij}-\sum K^*(\frac{\partial}{\partial y_i})K(\frac{\partial}{\partial \bar y_j})H_{ij}.$$
\end{thm}
\begin{proof}
By Lemma \ref{lem:connK}, $\sum [K,\tilde\nabla^{1,0}](\frac{\partial}{\partial y_i}, \frac{\partial}{\partial \bar y_j})H_{ij}=0$, $  \sum [\bpartial', K^*](\frac{\partial}{\partial y_i}, \frac{\partial}{\partial \bar y_j})H_{ij}$ also vanishes since it is the conjugacy of $\sum [K,\tilde\nabla^{1,0}](\frac{\partial}{\partial y_i}, \frac{\partial}{\partial \bar y_j})H_{ij}$. This proves the first equation by (\ref{eqn:theta}). By (\ref{eqn:curB1}), $\sum \tilde \Theta(\frac{\partial}{\partial y_i}, \frac{\partial}{\partial \bar y_j})H_{ij}$ vanishes if the mean curvatures of  $h_{T^*}$, $h_E$ and $h_F$ vanish. The second equation follows.
\end{proof}

\section{Holomorphic sections}
In this section, we assume $X$ is a compact Hermitian manifold. We calculate the space of holomorphic sections of $A_m(E,F)$ if the mean curvatures of  $h_{T^*}$, $h_E$ and $h_F$ are zero.
\begin{lemma}\label{lem:hol}Let $E$ be a holomorphic vector bundle  with a Hermitian metric $h=(-,-)$ and its Chern connection is  $\nabla=\nabla^{1,0}+\bpartial$.
 Assume $E$ has another holomorphic holomorphic structure determined by $\bpartial'=\bpartial-K$ with its Chern connection $\tilde \nabla=\tilde \nabla^{1,0}+\bpartial'$, such that under this holomorphic structure,
 \begin{enumerate}
 \item \label{enu:Hitem1}$E=\bigoplus_{k\in \mathbb Z}E^k$, $E^k$ are holomorphic vector bundles;
 \item\label{enu:Hitem2} $E^k$ is perpendicular to $E^l$ for $k\neq l$;
 \item \label{enu:Hitem3}the mean curvature of $\tilde \nabla$ is zero.
 \item \label{enu:Hitem4}$K$ maps $\Omega^{0,0}(E_l)$ to $\Omega^{0,1}(\oplus_{k=l+1}^{\infty}E_k)$.
 \item \label{enu:Hitem5}$\sum [\tilde\nabla^{1,0},K](\frac{\partial}{\partial y_i}, \frac{\partial}{\partial \bar y_j})H_{ij}=0.$
 \end{enumerate}
Then $a\in \Omega_X^{0,0}(E)$ is holomorphic if and only if $\tilde \nabla a=0$ and $Ka=0$.
\end{lemma}
\begin{proof}
By condition (\ref{enu:Hitem1}), any smooth section $a$ of $E$ can be uniquely written as a finite sum of smooth sections $a_k$ of $E^k$, that is  $a=\sum_k a_k$ where only a finite number of $a_k$  not zero. By condition (\ref{enu:Hitem2}), any smooth sections $a$ and $b$ of $E$,
$(a,b)=\sum_k (a_k,b_k)$. For any $\lambda\geq 0$, let $h^{\lambda}=(-,-)^{\lambda}$ be the Hermitian metric of $E$ given by
$$(a,b)^\lambda=\sum_k\lambda^k(a_k,b_k).$$
Under this Hermitian metric, Let $\nabla^{\lambda}$ and $\tilde \nabla^\lambda$ are the Chern connections corresponding to $\bpartial$ and $\bpartial'$ respectively. Since on $E^k$, $h^\lambda$ is a rescale of $h$, $\tilde \nabla^\lambda=\tilde \nabla$.

Let $K_{k,l}\in \Omega^{0,1}(\Hom(E, E))$ with $K_{k,l} a=(Ka_k)_l$. By condition (\ref{enu:Hitem4}),  $K=\sum_{k<l}K_{k,l}$. Let $K_{k,l}^*\in \Omega^{1,0}(\Hom(E, E))$ be the dual of $K_{k,l}$, so
$$(K_{k,l}(\frac{\partial}{\partial \bar y_i})a, b)^\lambda=\lambda^l(K_{k,l}(\frac{\partial}{\partial \bar y_i})a_k, b_l)=\lambda^l(a_k, K_{k,l}^*(\frac{\partial}{\partial  y_i})b_l)=\lambda^{l-k}(a, K_{k,l}^*(\frac{\partial}{\partial  y_i})b)^{\lambda}.$$
Let $$K^*_\lambda=\sum_{k<l}\lambda^{l-k}K_{k,l}^*,$$ we have $(K(\frac{\partial}{\partial \bar y_i})a, b)^\lambda=(a, K_{\lambda}^*(\frac{\partial}{\partial  y_i})b)^{\lambda}$.
By Lemma \ref{lem:connection}, $\nabla^\lambda=\tilde \nabla+K-K^*_\lambda$,
so its curvature is
\begin{eqnarray}
\Theta^\lambda=[\bpartial, \nabla]=[\bpartial'+K,\tilde\nabla^{1,0}-K^*_\lambda]=[\bpartial',\tilde\nabla]+[K,\tilde\nabla^{1,0}]-[\bpartial', K^*_\lambda]-[K,K^*_\lambda].\nonumber
\end{eqnarray}
By conditions (\ref{enu:Hitem3}) and (\ref{enu:Hitem5}) and the proof of Theorem \ref{thm:curv}, the mean curvature of $\nabla^\lambda$ is $$M=\sum \Theta^\lambda(\frac{\partial}{\partial y_i}, \frac{\partial}{\partial \bar y_j})H_{ij}=\sum K(\frac{\partial}{\partial \bar y_j})K^*_\lambda(\frac{\partial}{\partial y_i})H_{ij}-\sum K^*_\lambda(\frac{\partial}{\partial y_i})K(\frac{\partial}{\partial \bar y_j})H_{ij}.$$

Apply the Weitzenb\"ock formula \ref{eqn:weitz1} to a smooth section $a\in \Omega^{0,0}_X(E)$. We have
\begin{eqnarray*}\int_X(\Delta_{\bpartial}a^\lambda, a)^\lambda&=&\int_X(\nabla^{\lambda*}\nabla^\lambda a -M a,a)^\lambda\\
                 &=&\int_X(\nabla^\lambda a,\nabla^{\lambda}a)^\lambda-( K^*_\lambda a, K^*_\lambda a)^\lambda+ (K a,K a)^{\lambda}\\
                   &=&\int_X (\nabla^{\lambda 1,0} a,\nabla^{\lambda 1,0} a)^\lambda +(\bpartial a,\bpartial a)^\lambda
                                    -({K}^*_\lambda a,{K}^*_{\lambda}a)^{\lambda}
                        + (K a,K a)^{\lambda} .
\end{eqnarray*}

 Now  \begin{eqnarray}\label{equ:ka}\int_X(\bpartial a,\bpartial a)^\lambda\geq 0,\quad
                        \int_X(K a,K a)^{\lambda}\geq 0.
                        \end{eqnarray}

 Let
  $$P(\lambda)=\int_X (\nabla^{\lambda 1,0} a,\nabla^{\lambda 1,0} a)^\lambda -({K}^*_\lambda a,{K}^*_{\lambda}a)^{\lambda}. $$

  Assume $a_k=0$ for $k<l$ or $k>m$,   $a_l\neq 0$ and $a_m\neq 0$.
  \begin{eqnarray*}& &(\nabla^{\lambda 1,0} a,\nabla^{\lambda 1,0} a)^\lambda -({K}^*_\lambda a,{K}^*_{\lambda}a)^{\lambda}\\
&=&((\tilde \nabla^{1,0}-K^*_{\lambda})a,(\tilde\nabla^{1,0}-K^*_{\lambda})a)^{\lambda}-(K_\lambda^* a,K^*_{\lambda}a)^{\lambda}\\
 &=&(\tilde \nabla^{1,0}a,\tilde\nabla^{1,0}a)^{\lambda}-(\tilde \nabla^{1,0}a,K^*_{\lambda} a)^{\lambda}-(K^*_{\lambda} a,\tilde\nabla^{1,0}a)^{\lambda}.\\
 &=&\sum_{k=l}^m\lambda^l(\tilde \nabla^{1,0} a_l,\tilde \nabla^{1,0}a_l)
 -\sum_{m\geq k'>k\geq l}\lambda^{k'}(\tilde \nabla^{1,0}a_k,K^*_{k,k'}a_{k'})
 -\sum_{m\geq k'>k\geq l}(K^*_{k,k'}a_{k'},\tilde\nabla^{1,0}a_{k}).
 \end{eqnarray*}

 $\lambda^{-l}P(\lambda)$ is a polynomial of $\lambda$ with constant term
 $$\int_X(\tilde \nabla^{1,0}a_l,\tilde \nabla^{1,0}a_l)\geq 0.$$
 If $a$ is a holomorphic section, we have $\int_X(\Delta^\lambda_{\bpartial}a, a)^{\lambda}=0$ for any $\lambda>0$. By (\ref{equ:ka}), $P(\lambda)\leq 0$.
We must have
 $$\int_X(\tilde \nabla^{1,0}a_l,\tilde \nabla^{1,0}a_l)=0.$$
 So
  $$\tilde \nabla^{1,0}a_l=0.$$
  $K^*_{\lambda} a_l\in \Omega^{1,0}_X(\bigoplus_{k<l}E^k)$, which is perpendicular to $\tilde \nabla^{1,0} a_k$, $k\geq l$. Thus
  $$P(\lambda)=\int_X(\sum_{k=l+1}^m\lambda^l(\tilde \nabla^{1,0} a_l,\tilde \nabla^{1,0}a_l)
 -\sum_{m\geq k'>k\geq l+1}\lambda^{k'}(\tilde \nabla^{1,0}a_k,K^*_{k,k'}a_{k'})
 -\sum_{m\geq k'>k\geq l+1}(K^*_{k,k'}a_{k'},\tilde\nabla^{1,0}a_{k})).$$
Induction on $l$, we can show that $\tilde\nabla^{1,0}a=0$ and $P(\lambda)=0$.

  Thus to keep $\int_X(\Delta_{\bpartial}a, a)^{\lambda}=0$, the left side of the inequality (\ref{equ:ka}) is zero. So
  $$Ka=0, \quad \text{and }\,\bpartial a=0.$$
$\bpartial' a=\bpartial a-Ka=0$ and $\tilde \nabla a=\tilde \nabla^{1,0}a+\bpartial'a=0$.

On the other hand, suppose $\tilde \nabla a=0$ and $Ka=0$. Since the mean curvature of $\tilde \nabla$ is zero, by the Weitzenb\"ock formula (\ref{eqn:weitz1}), $\bpartial' a=0$.
Therefore $\bpartial a=\bpartial'a + K a=0$, so that $a$ is a holomorphic section of $E$.
\end{proof}

Let
\begin{eqnarray}\label{eqn:En}E^n=\Phi_m^{-1}
(\bigoplus_{j+k+l=n}\Sym^{j}(\oplus_{1\leq l\leq m} T^*)\otimes \Sym^k(\oplus_{0\leq l\leq m} E^*)\otimes\wedge^l(\oplus_{0\leq l\leq m} F^*)).
\end{eqnarray}
$A_m(E,F)=\oplus_{k=0}^\infty E^k$.
Let $A^*_m(E,F)[i,j,k]$ be the dual vector bundle of $A^*_m(E,F)[i,j,k]$. Let
$$A^*_m(E,F)=\oplus_{i,j,k}A^*_m(E,F)[i,j,k].$$
 We have the following theorem:
  \begin{thm} \label{thm:GE}If $X$ is a compact Hermitian manifold and the mean curvatures of  $h_{T^*}$, $h_E$ and $h_F$ are zero, then a smooth section $a$ of $A_m(E,F)$ is holomorphic if and only if $\tilde \nabla a=0$ and $Ka=0$. A smooth section $a$ of $A_m(E,F)^*$ is holomorphic if and only if $\tilde \nabla a=0$ and $K^{\vee} a=0$.
  \end{thm}
\begin{proof}We use Lemma \ref{lem:hol} to prove the theorem.  $A_m(E,F)$ is a holomorphic bundle with the Hermitian metric $h$ and its Chern connection is $\nabla$. $A_m(E,F)$ has another holomorphic structure coming from $B_m(E,F)$ by $\Phi_m$. Its Chern connection is $\tilde \nabla$.
$E^k$ are holomorphic under the holomorphic structure from $B_m(E,F)$. $E^k$ is perpendicular to $E^l$ if $k\neq l$.
$A_m(E,F)=\oplus_{k=0}^\infty E^k$. Let $E^{-k}=E^{k*}$, then $A^*_m(E,F)=\oplus_{-\infty}^0 E^{k}$. So $A_m(E,F)$ and $A^*_m(E,F)$ satisfies condition (1) and (2).
Since the mean curvature of $h_{T*}$, $h_E$ and $h_F$ vanish, by (\ref{eqn:curB1}), $A^*_m(E,F)$ satisfies condition (3).   By Lemma \ref{lem:Dbar} and Lemma \ref{lem:connK}, $A^*_m(E,F)$ satisfies conditions (4) and (5). Since $E^{-k}$ is the dual of $E^k$, it is easy to see that $A^*_m(E,F)$ satisfies conditions (3), (4) and (5). So by Lemma \ref{lem:hol}, the theorem is true.

\end{proof}

When the Hermitian metric of $E,F$ and $T^*$ are flat, then $K=0$ and  a smooth section $b$ of $A_m(E,F)^*$ is holomorphic if and only if $\tilde \nabla a=0$. In particular
\begin{cor}\label{cor:torus} If $X$ is a flat K\"ahler torus, E and F are sums of some copies of tangent and cotangent bundles of $X$, then a smooth section $a$ of $A_m(E,F)$ is holomorphic if and only if $\tilde \nabla a=0$. So the space of holomorphic sections of $A_m(E,F)$ is isomorphic to its fibre by restriction.
\end{cor}

\subsection{$\mathfrak g[t]$ invariants.}
Let $\mathfrak g$ be a Lie algebra. Let
$$\mathfrak g[t]=\bigoplus_{ n\geq 0} \mathfrak gt^n$$
 be the Lie algebra given by
$$[g_it^i, g_jt^j]=[g_i,g_j]t^{i+j}, \text{for } g_i, \, g_j\in \mathfrak g.$$
Let $\mathfrak g_m=\mathfrak g[t]/(t^{m+1})$.
The following lemma is obvious.
\begin{lemma} if $\mathfrak g$ is simple, then $\mathfrak g_m$ is generated by $\mathfrak g$ and an element $K$ of $\mathfrak gt$. So if $R$ is a representation of $\gg_m$, then the invariant subspace of $R$ is
$$R^{\gg_m} =\{ r\in R^{\gg}|Kr=0\}.$$
\end{lemma}
If $\mathfrak g$ is the Lie algebra of a connected algebraic group $G$, then $\mathfrak g_m$ is the Lie algebra of the Lie group $J_m(G)$. $R^{\gg_m}=R^{J_m(G)}$.

Let $G=SU(N)$ be the special unitary group on $V=\mathbb C^N$
and  $\mathfrak g=\mathfrak{sl}(N,\mathbb C)$ be its complexified Lie algebra (or $G=Sp(N)$ be the symplectic group on $V=\mathbb C^{2N}$ and $\mathfrak g=\mathfrak{sp}(2N,\mathbb C)$ be its complexified Lie algebra). By the Weyl's dimension formula for the finite irreducible representation of $\mathfrak g$, we have the following lemma about the representation of $\mathfrak g=\mathfrak{sl}(N,\mathbb C)$ (or $\mathfrak{sp}(2N,\mathbb C)$):
\begin{lemma}\label{lem:sun}

Let $V$ be the fundamental representation of $\mathfrak g$. Regarding $\mathfrak g$  as a subspace of $V^*\otimes V$,
$V^*\otimes \mathfrak g\bigcap \Sym^2V^*\otimes V$ is an irreducible representation of $\mathfrak g$.
\end{lemma}

  Now we assume $X$ is an $N$ (or $2N$) dimensional compact K\"ahler manifold with the holonomy group $G=SU(N) $ (or $G=Sp(N)$). Then the mean curvature of $h_{T^*}$ is zero. Let $E$ and $F$ be  sums of some copies of  the holomorphic tangent and cotangent bundles of $X$. Then the mean curvatures of $h_E$ and $h_F$ are zeros. The holonomy group of $A_m(E,F)$ for $\tilde \nabla $ is $G$. By Theorem \ref{thm:GE}, $a$ is a holomorphic section of $A_m(E,F)$ if and only if $\tilde \nabla a=0$ and $Ka=0$. $\tilde \nabla a=0$  means that $a$ is a parallel section of $\tilde \nabla$ in $A_m(E,F)$.
  The space of the parallel sections of $\tilde \nabla$ in $A_m(E,F)$ is isomorphic to  $(A_m(E,F)|_x)^{G}$, for any point  $x\in X$. The isomorphism is given by the restriction.
  So the restriction $$ r_x: \Gamma(A_m(E,F))\to (A_m(E,F)|_x)^G$$ is injective.
 Let $R=A_m(E,F)|_x$ be the fibre of $A_m(E,F)$ at $x$.
 Let $\mathfrak g$ be the Lie algebra  $\mathfrak {sl}(N,\mathbb C))$ (or $\mathfrak{sp}(2N,\mathbb C$ ). The action of  $G$ on $R$ induces the  action  of $\mathfrak g$  on $R$, then we have the action of $\mathfrak g_{m}$  on $R$ given by
 $$gt^k(ab)=(gt^ka) b+a (gt^kb),\quad \text{for }  a, b \in R$$
   and if  $ P_i^{(j)}$  is $  Y_i^{(j+1)}$,$E_i^{(j)}$ or $ F_i^{(j)}$,
 $$gt^k (P_i^{(j)})=0, \text{ for } j<k$$
 $$gt^k (P_i^{(j)})=\frac{j!}{(j-k)!} g (P_i^{(j-k)}), \text { for } j\geq k .$$
 We  have the following theorem on the holomorphic sections of $A_m(E,F)$.

 \begin{thm}\label{thm:holo} Let $X$ be an $N$ (or $2N$) dimensional compact K\"ahler manifold with holonomy group $G=SU(N) $ (or $G=Sp(N)$). Let $E$ and $F$ be  sums of some copies of  the holomorphic tangent and cotangent bundles of $X$. Then the image of $r_x: \Gamma(A_m(E,F))\to (A_m(E,F)|_x)^{G}$ is $(A_m(E,F)|_x)^{\mathfrak g_m}$. So $r_x$ induces an isomorphism from $\Gamma(A_m(E,F))$ to $(A_m(E,F)|_x)^{\mathfrak g_m}$.
 \end{thm}
 \begin{proof}
If $a$ is a parallel section of $\tilde\nabla$ with $r_x(a)\in (A_m(E,F)|_x)^{\mathfrak g_m}$ for a point $x \in X$. For any point $x'\in X$, let's prove that $r_{x'}(a)\in (A_m(E,F)|_{x'})^{\mathfrak g_m}$. Let $\gamma$ be a path from $x$ to $x'$. Through the parallel  transformation along the path $\gamma$ under the connection $\tilde \nabla$, we get an isomorphism of $G$ representation from $A_m(E,F)|_x$ to $A_m(E,F)|_{x'}$. From the definition of  the action $\mathfrak g_m$, we actually get an isomorphism of $\mathfrak g_m$ representation. Now since $a$ is a parallel section, the isomorphism maps $a|_x$ to $a|_{x'}$. Thus $a|_x$ is $\mathfrak g_m$ invariant if and only if $a|_{x'}$ is $\mathfrak g_m$ invariant.

 The holonomy group of $A_m(E,F)$ for $\tilde \nabla $ is $G$. Through the action of $G$ on $A_m(E,F)|_x$, we can regard $\mathfrak g$ as a subspace of $ \End(A_m(E,F)|_x, (A_m(E,F)|_x)$. We have
 $\tilde \Theta(\frac{\partial}{\partial y_i},\frac{\partial}{\partial \bar y_j}) |_x\in \mathfrak g$.
Let $\tilde F_n=F_n^{A_m(E,F)}$ be the higher covariant derivative of $\tilde \Theta$. Then
 $$g_{i_2\cdots i_a,j}=\tilde F_n(\frac {\partial}{\partial  y_{i_2}},\cdots,\frac{\partial}{\partial  y_{i_{n}}},\frac {\partial}{\partial  \bar y_{j}})|_x\in\mathfrak g.$$
By the expression of $K$ in Lemma \ref{lem:Dbar},
\begin{eqnarray*}
K(\frac {\partial}{\partial  \bar y_{j}}) P_l^{(k)}|_x&=&\sum_{a=2}^k\sum_{j_1+\cdots+j_a=k}\sum_{i_2,\cdots,i_{a}}C_{j_1,\cdots,j_a} (\tilde F_n(\frac {\partial}{\partial  y_{i_2}},\cdots,\frac{\partial}{\partial  y_{i_{a}}},\frac {\partial}{\partial  \bar y_{j}})|_xP_{l}^{(k-j_1)})Y_{i_2}^{(j_2)}\cdots Y_{i_a}^{(j_a)}\\
&=&\sum_{a=2}^k\sum_{j_1+\cdots+j_a=k}\sum_{i_2,\cdots,i_{a}}\frac{1}{(a-1)!j_2!\cdots j_a!}( g_{i_2\cdots i_a,j}t^{k-j_1}P_{l}^{(k)})Y_{i_2}^{(j_2)}\cdots Y_{i_a}^{(j_a)}
\end{eqnarray*}
So
\begin{eqnarray}\label{eqn:Kg}K(\frac {\partial}{\partial \bar y_i})|_x=\sum_{a=2}^k\sum_{j_1+\cdots+j_a=k}\sum_{i_2,\cdots,i_{a}}\frac{1}{(a-1)!j_2!\cdots j_a!}Y_{i_2}^{(j_2)}\cdots Y_{i_a}^{(j_a)}g_{i_2\cdots i_a,j}t^{k-j_1}.
\end{eqnarray}
Since $a|_x$ is $\mathfrak g_m$ invariant for any $x\in X$, so $K(a)=0$. By Theorem \ref{thm:GE},  $a\in \Gamma(A_m(E,F))$.

On the other hand, if $a \in \Gamma(A_m(E,F))$, $a=\sum_{k=l}^n a_k$ with $a_k\in \Omega^{0,0}(E^k)$.
Let
$$\tilde  K_i=\sum_{j_2\geq 1,i_2}\frac{Y_{i_2}^{(j_2)}}{j_2!}\tilde \Theta(\frac{\partial}{\partial y_{i_2}}, \frac{\partial}{\partial \bar y_{i}})t^{j_2}.$$
Then
$$K(\frac {\partial}{\partial  \bar y_{i}})a|_x=\tilde K_i a_l|_x+b, \quad  b\in \oplus_{k>l}E^{k}|_x.$$
So $Ka=0$ implies $\tilde K_ia_l|_x=0$.

Since $\tilde \nabla $ preserves $E^k$, the action of the holomomy group $G$ on the fibre of $A_m(E,F)$ preserves the grading. So $a_l|_x\in  (E^l|_x)^{G}$ .
Since the holonomy group of $X$ is $G$, the curvature $R$ of $T$ is not equals to zero. There is some point $x\in X$ , and some $i_0$,  $R(\frac {\partial}{\partial  \bar y_{i_0}})|_p\neq 0$.

The action of $\mathfrak g$  on $K_{i_0}$ forms a representation $W$ of $\mathfrak g$. Every element of $W$ is a derivation of $\Omega^{0,0}(E_m)$ and it vanishes on $a_l|_p$ .

Now for any  $g\in \mathfrak g$.
\begin{eqnarray*}[g, \tilde K_{i_0}]&=&\sum_{j_2\geq 1,i_2}\big((g\frac{Y_{i_2}^{(j_2)}}{j_2!})(\tilde \Theta (\frac {\partial}{\partial  \bar y_{i_0}}, \frac {\partial}{\partial y_{i_2}})t^{j_2})+\frac{Y_{i_2}^{(j_2)}}{j_2!} ([g,\tilde \Theta(\frac {\partial}{\partial  \bar y_{i_0}}, \frac {\partial}{\partial y_{i_2}})]t^{j_2})\big)\\
&=& \sum_{j_2\geq 1,i_2}\big(-\frac{Y_{i_2}^{(j_2)}}{j_2!}(\tilde \Theta(\frac {\partial}{\partial  \bar y_{i_0}}, g\frac {\partial}{\partial y_{i_2}})t^{j_2})+\frac{Y_{i_2}^{(j_2)}}{j_2!} ([g,\tilde\Theta(\frac {\partial}{\partial  \bar y_{i_0}}, \frac {\partial}{\partial y_{i_2}})]t^{j_2})
\end{eqnarray*}
Thus the action of $\mathfrak g$  on $K_{i_0}$  is the same as the action of $\mathfrak g$ on $\tilde \Theta(\frac {\partial}{\partial  \bar y_{i_0}})$. Since $E$ and $F$ are direct sums of copies of holomorphic tangent and cotangent bundles. By (\ref{eqn:curB1}), the action of $\mathfrak g$ on $\tilde \Theta(\frac {\partial}{\partial  \bar y_{i_0}})$ is the same as  its action on
$R(\frac {\partial}{\partial  \bar y_{i_0}})\in \Sym^2 T_x^*\otimes T_x\bigcap T^*_x\otimes \mathfrak g$. By Lemma \ref{lem:sun}, $\Sym^2 T_x^*\otimes T_x\bigcap T^*_x\otimes \mathfrak g$ is an irreducible representation of $\mathfrak g$. So
$$W\cong \Sym^2 T_x^*\otimes T_x\bigcap T^*_x\otimes \mathfrak g.$$
Let $g_1, g_2\in \mathfrak g$, which are determined  by
$$g_1 (\frac {\partial}{\partial y_{i}})=\delta_{i}^1\frac{\partial}{\partial y_{s}},\quad g_2 (\frac {\partial}{\partial y_{i}})=\delta_{i}^1\frac{\partial}{\partial y_{1}}-\delta_{i}^s\frac{\partial}{\partial y_{s}}.$$
here $s=2$ if $g=\mathfrak{sl}(N,\mathbb C)$ and $ s=N+1$ if $g=\mathfrak{sp}(2N,\mathbb C)$.
We have $[g_2,g_1]=-2g_1$.
Then the two derivations
$$K_1=\sum_{j\geq 1}\frac{Y_{1}^{(j)}}{j!}g_1 t^{j},\quad K_2=-\sum_{j\geq 1}\frac{Y_{s}^{(j)}}{j!}g_1 t^{j}+\sum_{j\geq 1}\frac{Y_{1}^{(j)}}{j!}g_2 t^{j}$$
are in $W$. We have $K_1 a_l|_x=0$ and $K_2 a_l|_x=0$.
By Lemma \ref{lem:cri}, $a_l|_x$ is $g_m$ invariant. Then we have $K a_l=0$ . Now  $K (a-a_l)=0$. By induction on $l$,
we can show $a_x$  is $\mathfrak g_m$ invariant. This proves the theorem.

\end{proof}

\begin{lemma}\label{lem:cri} if $a\in A_m(E,F)|_x$ is $\mathfrak g$ invariant and $K_1a=0$, $K_2a=0$, then $a$ is $\mathfrak g_m$ invariant.
\end{lemma}
\begin{proof}
Let $\mathbb P_n$ be a derivation on $ A_m(E,F)|_x$  defined inductively by
$$\mathbb P_1=K_1,\quad \mathbb P_n=[\mathbb P_{n-1}, K_2].$$
We have $\mathbb P_n a=0$. Now
$$[\frac{Y_{1}^{(j)}}{j!}, K_2]=\frac 1 j\sum_{s=1}^{j-1}\frac{Y_{1}^{(j-s)}}{(j-s-1)!)}\frac{Y_{1}^{(s)}}{s!},$$
$$[g_1t^j, K_2]=2 \sum_{s=1}^{j-1}\frac{Y_{1}^{(s)}}{s!} g_1t^{j+s}.$$
By the above two equations, we can show inductively that
$$\mathbb P_n=\sum_{j_i\geq 1} c_{j_1,\cdots, j_a} Y_1^{(j_1)}\cdots Y_1^{(j_n)}g_1t^{j_1+\cdots+j_n}$$
with $c_{j_1,\cdots, j_n}>0$.
$a$ has finite weight, when $k$ large enough $gt^k a=0$ for any $g\in \mathfrak g$. Let $L$ be the largest number such that $a$ is not $\mathfrak gt^L a$ invariant. if $L\geq 1$,
$$\mathbb P_L a=c_{1,\cdots,1}(Y_1^{(1)})^Lg_1t^L a.$$
So $g_1t^L a=0$. Since $a$ is $\mathfrak g$ invariant and $\mathfrak gt^L$ is an irreducible representation of $\mathfrak g$, $a$ is $\mathfrak gt^L$ invariant. But we assume that  $a$ is not $\mathfrak gt^L$ invariant. So $L\leq 0$.
we conclude that $a$ is $\mathfrak g_m$ invariant.
\end{proof}
Now we can prove Theorem \ref{thm:TJ}.
\begin{proof}[Proof of theorem \ref{thm:TJ}]Let  $\mathfrak g=\mathfrak{sl}(N,\mathbb C)$.
 $TJ_m(X)$ is isomorphic to $ J_m(TX)$. The sheaf of sections of $A_m(T,0)[1,0]$ over $X$ is the sheaf of the push forward of the
sheaf of sections of  $J_m(TX)$ through $\pi_m: J_m(X)\to X$. Thus the space of holomorphic vector fields of $J_m(X)$ is isomorphic to $\Gamma(A_m(T,0)[1,0])$. By Theorem \ref{thm:holo},
$\Gamma(A_m(T,0))$ is isomorphic to $(A_m(T,0)|_x)^{\mathfrak g_m}$.
$$A_m(T,0)|_x\cong R(y,y^*,0)= \mathbb C[y_1^{(1)}, \cdots, y_i^{(j)},\cdots,y_N^{(m)},y_1^{*(0)}, \cdots, y_i^{*(j)},\cdots,y_N^{*(m)}],$$
 So $\Gamma(A_m(T,0)[1,0])$ is isomorphic to the subspace of $(A_m(T,0)|_x)^{\mathfrak g_m}$ with the degree of $y^*$ is one.
 By Theorem 4.3 of \cite{LSSI},  $(A_m(T,0)|_x)^{\mathfrak g_m}$ is generated by $ \mathbb C[y_1^{(1)}, \cdots, y_N^{(1)},y_1^{*(0)}, \cdots, y_N^{*(0)}]^{\mathfrak g}$, which is generated by
 $v=\sum y_i^{(1)}y_i^{*(0)}$. So
 $$\Gamma(A_m(T,0)[1,0])\cong \bigoplus_{k=0}^{m-1}\mathbb C \tilde D^kv,$$
 has dimension $m$.
\end{proof}

Theorem \ref{thm:holo} can be generalized to the case when $X$ is a compact Ricci-flat K\"ahler manifold.

 \begin{thm}\label{Thm:holoX} Let $X$ be  compact Ricci-flat K\"ahler manifold with holonomy group $G$. Let $\mathfrak g$ be the complexified Lie algebra of $G$. Let $E$ and $F$ be  sums of some copies of  the holomorphic tangent and cotangent bundles of $X$. Then the image of $r_x: \Gamma(A_m(E,F))\to (A_m(E,F)|_x)^{G}$ is $((A_m(E,F)|_x)^{\mathfrak g_m})^G$. So $r_x$ induces an isomorphism from $\Gamma(A_m(E,F))$ to $((A_m(E,F)|_x)^{\mathfrak g_m})^G$.
 \end{thm}

 \begin{proof}
 It is well known that if $X$ is a compact Ricci-flat K\"ahler manifold, then it admits a finite cover  $p:\tilde X\to X$,
$$\tilde X=T\times X_1\times\cdots\times X_k,$$ where $T$ is a flat K\"ahler torus, and $X_i$ has holonomy group $SU(m_i)$ or $Sp(\frac {m_i} 2)$ with $\dim X_i=m_i$. (See for example page 124 of \cite{J}.)
 If $E$ and $F$ are sums of some copies of  the holomorphic tangent and cotangent bundles of $X$, then the pullback $p^*E$ and $p^*F$ are sums of some copies of  the holomorphic tangent and cotangent bundles of $\tilde X$. We have
 $$p^* A_m(E,F)=A_m(p^*E,p^*F).$$
  By Theorem \ref{thm:holo}, Corollary \ref{cor:torus} and the fact that holonomoy group of $T$ is trivial,   $r_{\tilde x}$ induces an isomorphism from $\Gamma(A_m(p^*E,p^*F))$ to $(A_m(p^*E,p^*F)|_{\tilde x})^{\mathfrak g_m}$.

  The pullback gives an imbedding $p^*: \Gamma(X, A(E,F))\to \Gamma(\tilde X, A(p^*E,p^*F)$. By Theorem \ref{thm:GE}, any element $\tilde a\in \Gamma(\tilde X, A(p^*E,p^*F)$  satisfies  $\tilde \nabla a=0$. So  $\tilde a=p^*a$ for some smooth section $a$ of $A(E,F)$ if and only if $a|_{\tilde x}$ is $G$ invariant. Since $\tilde a$ is holomorphic, $a$ must be holomorphic. Thus $r_{\tilde x} \circ p^*$ induces an isomorphism from $\Gamma(X, A(E,F))$ to $((A_m(p^*E,p^*F)|_{\tilde x})^{\mathfrak g_m})^G$, which means $r_x$ induces an isomorphism from $\Gamma(A_m(E,F))$ to $((A_m(E,F)|_x)^{\mathfrak g_m})^G$.
\end{proof}

$G$ has a connected component $G_0$ which contain the identity. $G_0$ is a normal subgroup of $G$ and it  acts trivially on  $(A_m(E,F)|_x)^{\mathfrak g_m}$. So  $((A_m(E,F)|_x)^{\mathfrak g_m})^G= ((A_m(E,F)|_x)^{\mathfrak g_m})^{G/G_0}$
\section{The global sections of chiral de Rham complex on K3 surfaces}

In this paper, we will follow the formalism of the vertex algebra developed in \cite{K}. We also use the notation in our earlier paper \cite{LSSII} and \cite{S}.
\subsection{Chiral de Rham complex}
The chiral de Rham complex~\cite{MSV, MS} is a sheaf of vertex algebras $\Omega_X^{ch}$ defined on any smooth manifold $X$ in either the algebraic, complex analytic, or $C^{\infty}$ categories. In this paper we work exclusively in the complex analytic setting. For the vertex algebra, we refer \cite{K}.

Let $\Omega_N$ be the tensor product of $N$ copies of the $\beta\gamma-bc$ system. It has $2N$ even generators $\beta^1(z),\cdots ,\beta^N(z),\gamma^1(z),\cdots,\gamma^N(z)$ and $2N$ odd generators $b^1(z),\cdots ,b^N(z), c^1(z),\cdots,c^N(z)$.
Their nontrivial OPEs are $$\beta^i(z) \gamma^j(w)\sim \frac {\delta^i_j}{z-w},\quad b^i(z) c^j(w)\sim \frac {\delta^i_j}{z-w} .$$

Given a coordinate system $(U,\gamma^1,\cdots \gamma^N)$ of $X$,
$\mathbb C[\gamma^1,\cdots \gamma^N]\subset \mathcal O(U)$ can be regarded as a subspace of $\Omega_N$
by identifying $\gamma^i$ with  $\gamma^i(z)\in \Omega_N$. As a linear space, $\Omega_N$ has a $\mathbb C[\gamma^1,\cdots \gamma^N]$ module structure. $\Omega_X^{ch}(U)$ is the localization of $\Omega_N$ on
$U$, $$\Omega_X^{ch}(U)=\Omega_N\otimes_{\mathbb C[\gamma^1,\cdots \gamma^N]}\mathcal O(U).$$
Then $\Omega_X^{ch}(U)$ is the vertex algebra generated by $\beta^i(z), b^i(z), c^i(z)$ and $f(z)$, $f\in \mathcal O(U)$. These generators satisfy the nontrivial OPEs
$$\beta^i(z)  f(w)\sim \frac {\frac{\partial f}{\partial \gamma^i}(z)}{z-w},\quad b^i(z) c^j(w)\sim \frac {\delta^i_j}{z-w},$$ as well as the normally ordered relations $$:f(z)g(z):\ =fg(z), \text{ for }  f, g \in \mathcal O(U).$$
 $\Omega_X^{ch}(U)$ is spanned by the elements
\begin{equation}\label{eq.span}
:\partial^{k_1}\beta^{i_1}(z)\cdots \partial^{k_s}\beta^{i_s}(z)\partial^{l_1}b^{j_1}(z)
\cdots \partial^{l_t}b^{j_t}\partial^{m_1}c^{r_1}(z)
 \cdots\partial^{n_1}\gamma^{s_1}(z)\cdots f(\gamma)(z): , \quad f(\gamma)\in \mathcal O(U) .
 \end{equation}
 Let $\tilde \gamma^1,\cdots \tilde \gamma^N$ be another set of coordinates on $U$, with
$$\tilde \gamma^i=f^i(\gamma^1,\cdots \gamma^N), \quad \gamma^i=g^i(\tilde \gamma^1,\cdots \tilde \gamma^N).$$
The coordinate transformation equations for the generators are
\begin{align}\label{chi.coo}
\partial \tilde \gamma^i(z)&=\sum :\frac{\partial f^i}{\partial \gamma^j}(z)\partial \gamma^j(z):\,, \nonumber \\
\tilde b^i(z)&=\sum :\frac{\partial g^j}{\partial \tilde \gamma^i}(g(\gamma))b^j: \nonumber\,, \\
\tilde c^i(z)&=\sum :\frac{\partial f^i}{\partial \gamma^j}(z)c^j(z):\,,\\
\tilde \beta^i(z)&=\sum :\frac{\partial g^j}{\partial \tilde \gamma^i}(g(\gamma))(z)\beta^j(z):
+\sum ::\frac{\partial}{\partial \gamma^k}(\frac{\partial g^j}{\partial \tilde \gamma^i}(g(\gamma)))(z)c^k(z):b^j(z):\,.\nonumber
\end{align}




From \cite{S}, $\Omega_X^{ch}$ has an increasing filtration $\mathcal Q_n$, $n\in \mathbb Z_{\geq 0}$ and its associated graded sheaf is
$$\text{gr}(\cQ)=\bigoplus_n \cQ_n/\cQ_{n-1}.$$
Locally,
 $\cQ_n(U)$ is spanned by the elements with only at most $n$ copies of $\beta$ and $b$ , i.e. the elements in equation~\eqref{eq.span} with $s+t\leq n$.

Then the associated graded object
$$(\text{gr}\cQ)(U)=\bigoplus_n \cQ_n(U)/\cQ_{n-1}(U)$$
is a $\partial$-ring. A $\mathbb Z_{\geq 0}$ graded, associative, super-commutative algebra equipped with a derivation $\partial$ of degree zero is called
$\partial$-ring. On $(\text{gr}\cQ)(U)$, the product and the derivation $\partial$ are induced from the wick product and $\partial$
from $\Omega_X^{ch}(U)$, respectively.

For each $n\geq 0$, let  $$\phi_n:\cQ_n(U )\to \cQ_n(U)/\cQ_{n-1}(U).$$ be the projection.
As a ring with a derivation, $\text{gr}(Q(U))$ is generated by
$$\beta^i=\phi_1(\beta^i(z)), \quad b^i=\phi_1(b^i(z)),\quad c^i= \phi_0(c^i(z)),\quad \phi_0(f(z)),\, f\in \cO(U).$$

 $\cQ_n(U)/\cQ_{n-1}(U)$ is spanned by
\begin{equation} \label{bas:a}
a=\partial^{k_1}\beta^{i_1}\cdot\cdot\cdot\partial^{k_s}\beta^{i_s}
\partial^{k_{s+1}}b^{i_{s+1}}\cdot\cdot\cdot \partial^{k_{n}}b^{i_{n}}\partial^{m_1}c^{r_1}
 \cdot\cdot\cdot\partial^{n_1}\gamma^{s_1}\cdot\cdot\cdot \partial^{n_t}\gamma^{s_t}f(\gamma),\quad f(\gamma)\in \mathcal O(U).
 \end{equation}

For the sheaf $\text{gr}(\cQ)$, the coordinate transformation equations of $\beta, \gamma, b, c $ are
\begin{align}\label{chi.co2}
\partial^n \tilde \gamma^i&=\sum \partial^{n-1}(\frac{\partial f^i}{\partial \gamma^j}\partial \gamma^j),\nonumber\\
\partial^n \tilde b^i&=\sum \partial^n(\frac{\partial g^j}{\partial \tilde \gamma^i}(g(\gamma))b^j),\nonumber\\
\partial^n \tilde c^i&=\sum \partial^n(\frac{\partial f^i}{\partial \gamma^j}c^j),\\
\partial^n \tilde \beta^i&=\sum \partial^n(\frac{\partial g^j}{\partial \tilde \gamma^j}(g(\gamma))\beta^i)+\sum 
\partial^n(\frac{\partial}{\partial \gamma^k}(\frac{\partial g^j}{\partial \tilde \gamma^i}(g(\gamma)))c^kb^j).\nonumber
\end{align}
 The only difference between these equations and those of the chiral de Rham sheaf is that the Wick product
 is replaced by the ordinary product in an associated super commutative algebra.

 $\text{gr}(\cQ)$ has an increasing filtration $\text{gr}(\cQ)_{n,s}$, $0\leq s\leq n$ and its associated graded sheaf is
 $$\text{gr}^2(\cQ)=\bigoplus_{n,s} \text{gr}(\cQ)_{n,s}/\text{gr}(\cQ)_{n,s-1}.$$

 Locally,   $\text{gr}(\cQ)_{n,s}(U)$ is spanned
by all elements $a\in \cQ_n(U)/\cQ_{n-1}(U)$ of the form \eqref{bas:a} with the number of $\beta$ less or equal than $s$.

The associated graded object
$$\text{gr}^2(\cQ)(U)=\bigoplus_{n,s} \text{gr}(\cQ)_{n,s}(U)/ \text{gr}(\cQ)_{n,s-1}(U)$$
is a $\partial$-ring. The product and the derivation $\partial$ are
 induced from the product and derivation of $\text{gr}(\cQ)(U)$.

 Let
$$\psi_{n,s}: \text{gr}(\cQ)_{n,s}(U)\to \text{gr}(\cQ)_{n,s}(U)/ \text{gr}(\cQ)_{n,s-1}(U)$$
be the projection as a ring with a derivation,$\text{gr}^2(Q(U))$ is generated by
 $$\psi_{1,1}(\beta^i),\psi_{1,0}(b^i), \psi_{0,0}(c^i)\,\text{ and } \psi_{0,0}(f(\gamma)),\,f(\gamma) \in \mathcal O(U).$$
When no confusion arise, the symbols $\beta^i, \gamma^i, b^i, c^i$ in $\text{gr}(\cQ)(U)$ will also be used to denote the corresponding elements $\psi_{1,1}(\beta^i)$, $\psi_{0,0}(\gamma^i)$, $\psi_{1,0}(b^i), \psi_{0,0}(c^i)$ in $\text{gr}^2(\cQ)(U)$.

For the sheaf $\text{gr}^2(\cQ)$, the  relations of $\beta, \gamma, b, c $ under the coordinate transformation are
\begin{align}\label{chi.co3}
\partial^n \tilde \gamma^i&=\sum \partial^{n-1}(\frac{\partial f^i}{\partial \gamma^j}\partial \gamma^j),\nonumber\\
\partial^n \tilde b^i&=\sum \partial^n(\frac{\partial g^j}{\partial \tilde \gamma^i}(g(\gamma))b^j),\nonumber\\
\partial^n \tilde c^i&=\sum \partial^n(\frac{\partial f^i}{\partial \gamma^j}c^j),\\
\partial^n \tilde \beta^i&=\sum \partial^n(\frac{\partial g^j}{\partial \tilde \gamma^i}(g(\gamma))\beta^j).\nonumber
\end{align}
By these coordinate transformation equations, we have
\begin{prop}\label{prop:iso}$\text{gr}^2(\cQ)$ is exactly the sheaf  $\mathcal A_\infty(T,T\oplus T^*)$.
\end{prop}

From \cite{S}, we have the following \textsl{reconstruction properties} for the holomorphic sections of $\Omega_X^{ch}$.
\begin{lemma}\label{lem.q}
If $a_i\in \cQ_{n_i}(X)$, for $1\leq i\leq l$, such that $\phi_{n_i}(a_i)$ generate $\text{gr}(\cQ)(X)$ as a $\partial$-ring, then
$\phi_n$ is surjective, and it therefore induces an isomorphism
$$\cQ_{n}(M)/\cQ_{n-1}(X){\cong} \cQ_{n}/\cQ_{n-1}(X).$$
Furthermore, $\{a_i|\ 1\leq i\leq l\}$ strongly generates the vertex algebra $\Omega_X^{ch}(M)$. \end{lemma}

\begin{lemma} \label{lem.q1}
If $a_i\in \text{gr}(\cQ)_{n_i,s_i}$ for $1\leq i\leq l$, such that $\psi_{n_i,s_i}(a_i)$ generate $\text{gr}(\cQ)(X)$ as a $\partial$-ring, then
$\psi_{n,s}$ is surjective, and it induces an isomorphism
$$\text{gr}(\cQ)_{n,s}(X)/\text{gr}(\cQ)_{n,s-1}(X){\cong} \text{gr}(\cQ)_{n,s}/\text{gr}(\cQ)_{n,s-1}(X).$$
Furthermore, $\{a_i|\ 1\leq i\leq l\}$ generate $\text{gr}(\cQ)(X)$ as a $\partial$-ring.
\end{lemma}

\subsection{Holomorphic sections of the chiral de Rham complex}
If $X$ is a Calabi-Yau manifold, according to \cite{MSV}, $\Omega_X^{ch}(X)$ is a topological vertex algebra. There are four global sections $Q(z), L(z), J(z), G(z)$.
 Locally, they can be represented by
\begin{align}
 &Q(z)= \sum_{i=1}^N:\beta^i(z)c^i(z):,& &L(z)=\sum_{i=1}^N(:\beta^i(z)\partial\gamma^i(z):-:b^i(z)\partial c^i(z):),&\nonumber\\
&J(z)=-\sum_{i=1}^N:b^i(z)c^i(z):,& &G(z)=\sum_{i=1}^N:b^i(z)\partial\gamma^i(z):,&\nonumber
\end{align}
According to \cite{EHKZ}, if $\omega$ is a nowhere vanishing holomorphic $N$ form  of $X$, there are two global sections $D(z)$ and $E(Z)$  of  $\Omega_X^{ch}$ can be constructed.
Locally, if $(U,\gamma)$ are coordinate of $U$ such that $\omega=d\gamma^1\wedge\cdots \wedge d\gamma^N$,  $D(z)$ and $E(Z)$ can be represented by
 $$D(z)=  :b^1(z)b^2(z)\cdots b^N(z):,\quad E(z)=:c^1(z)c^2(z)\cdots c^N(z):.$$
 Let $B(z)=Q(z)_{(0)}D(z)$,  $C(z)=G(z)_{(0)}E(z)$. It is easy to see that
 these eight sections $Q(z)$, $L(z)$, $J(z)$, $G(z)$, $D(z)$, $E(z)$, $B(z)$ and $C(z)$ always close (nonlinearly)
 under operator product expansion.
  Let $\mathcal V_N$ be the vertex algebra generated by these eight sections.  When $N=2$, $\mathcal V_2$ is an $\mathcal N=4$ superconfomal vertex algebra with central charge $6$. (Here the Virasoro field is $L(z)-\frac 1 2 \partial J(z)$.) We will show that if $X$ is a K3 surface, these eight sections strongly generate $\Omega_X^{ch}(X)$ as a vertex algebra.

Let $\bar \Omega_N$ be the subalgebra of the  $\Omega_N$, which is generated by $\beta^i(z),\partial \gamma^i(z), b^i(z), c^i(z)$. It is a tensor product of a system of free bosons and a system of free fermions. $\mathcal V_0$ is a subalgebra of $\bar \Omega_N$. On $\bar \Omega_N$ there is a positive definite Hermitian form $(-,-)$ with the following property:
\begin{eqnarray}\label{eqn:propherm}
&(\beta_{(n)}^{x_i}A,B)=(A,\alpha^{x'_i}_{(-n)}B), &\text{for any } n\in\mathbb Z , n\neq 0, \forall A, B\in \bar \Omega_N; \\
&(b^{x_i}_{(n)}A,B)=(A, c^{x_i'}_{(-n-1)}B), &\text{for any } n\in\mathbb Z, \forall A, B\in  \bar \Omega_N.\nonumber
\end{eqnarray}
It is easy to check
\begin{align}\label{eqn:conjugate}
&Q_{(n)}^*= G_{(-n+1)},& &J_{(n)}^*= J_{(-n)},&\\
&L_{(n)}^*= L_{(-n+2)}-(n-1)J_{(-n+1)},& &D_{(n)}^*= (-1)^{\frac {N(N-1)}2 }E_{(N-2-n)},&\nonumber
\end{align}
Let $\tilde L=L-\frac 1 2 \partial J$. $\tilde L$ is a Virasoro field with central charge $3N$.  $\tilde L_{(n)}^*=\tilde L_{(-n+2)}$.  We can conclude
\begin{lemma}\label{lem:simple}
 $\bar \Omega_N$ is a unitary representation of the Lie algebra generated by
  $$\{Q_{(n)}, G_{(n)}, J_{(n)},L_{(n)}, D_{(n)}, E_{(n)},B_{(n)}, C_{(n)}\}_{n\in \mathbb Z}$$
   and $\mathcal V_N$ is a simple conformal vertex algebra with central charge $3N$.
\end{lemma}
Assume $X$ is an $N$ dimensional complex manifold with holonomy group $SU(N)$.
By Theorem \ref{thm:holo} and Proposition \ref{prop:iso}, $\text{gr}^2(\cQ)(X)$ is isomorphic to $\text{gr}^2(\cQ)|_x^{\mathfrak{sl}(N,\mathbb C)[t]}$.
$ \text{gr}^2(\cQ)|_x$ is a $\partial$-ring, which is isomorphic to
$$R=\mathbb C[\partial^k\beta^i,\partial^{k+1}\gamma^i,\partial^kb^i,\partial^kc^i], \quad k\geq 0, 1\leq i\leq N.$$
So to calculate the holomorphic sections of $\text{gr}^2(\cQ)$ is to calculate $R^{\mathfrak{sl}(N,\mathbb C)[t]}$. It is easy to see that the following
eight elements of $R$ are $\mathfrak{sl}(N,\mathbb C)[t]$ invariant.
\begin{align}\label{equ:inv}
 &\sum_{i=1}^N\beta^ic^i,&  &\sum_{i=1}^N\beta^i\partial\gamma^i,&\nonumber\\
   &-\sum_{i=1}^N b^ic^i,&  &\sum_{i=1}^Nb^i\partial\gamma^i,&\nonumber\\
   & b^1b^2\cdots b^N,&   &c^1c^2\cdots c^N,&\\
   &\sum_{i=1}^N(-1)^{i-1} b^1\cdots b^{i-1}\beta^ib^{i+1}\cdots b^N,&  &\sum_{i=1}^N(-1)^{i-1} c^1\cdots c^{i-1}\partial\gamma^ic^{i+1}\cdots c^N.&\nonumber
\end{align}

 We have the following theorem.
\begin{thm}\label{thm:chiral}
If the holonomy group of the $N$-dimensional K\"ahler manifold $X$ is $SU(N)$ and $R^{\mathfrak{sl}(N,\mathbb C)[t]}$ is generated by above eight elements as a $\partial$-ring,
 then the  space of global sections of the chiral de Rham complex of  $X$ is  a simple vertex algebra strongly generated by
$$Q(z), L(z), J(z), G(z), B(z), D(z), C(z), E(z).$$
\end{thm}
\begin{proof}

 The images of the above eight sections under the map of $\phi_*$ and $\psi_{*,*}$ locally are
 \begin{align}\label{equ:gl2}
 &\psi_{1,1}(\phi_1(Q(z)))=\sum_{i=1}^N\beta^ic^i,&  &\psi_{1,1}(\phi_1(L(z)))=\sum_{i=1}^N\beta^i\partial\gamma^i,&\nonumber\\
   &\psi_{1,0}(\phi_1(J(z)))=-\sum_{i=1}^N b^ic^i,&  &\psi_{1,0}(\phi_1(G(z)))=\sum_{i=1}^Nb^i\partial\gamma^i,&\nonumber\\
   &\psi_{N,0}(\phi_N(D(z)))=\frac 1 f b^1b^2\cdots b^N,&   &\psi_{0,0}(\phi_0(E(z)))=fc^1c^2\cdots c^N,&\\
   &\psi_{N,1}(\phi_N(B(z)))=\sum_{i=1}^N(-1)^{i-1}\frac 1 f b^1\cdots b^{i-1}\beta^ib^{i+1}\cdots b^N,&  & &\nonumber \\
   & \psi_{0,0}(\phi_0(C(z)))=\sum_{i=1}^N(-1)^{i-1}f c^1\cdots c^{i-1}\partial\gamma^ic^{i+1}\cdots c^N.&  & &\nonumber
\end{align}

 They are global sections of $\text{gr}^2(Q)$. By assumption, the eight elements of (\ref{equ:inv}) generated $\text{gr}^2(\cQ)|_p^{\mathfrak{sl}(N,\mathbb C)[t]}$ as a $\partial$-ring.
  Compare (\ref{equ:inv}) and (\ref{equ:gl2}). By Theorem \ref{thm:holo} the eight global sections in (\ref{equ:gl2}) generate $\text{gr}^2(\cQ)(X)$ as a $\partial$-ring.
By Lemma \ref{lem.q} and \ref{lem.q1},
 the eight sections $Q(z), L(z), J(z), G(z), B(z), D(z), C(z), E(z)$ strongly generate the vertex algebra $\Omega_X^{ch}(X)$. By Lemma \ref{lem:simple}, this vertex algebra is simple.
\end{proof}

From \cite{LSSII} and the classical invariant theory \cite{We}, when $\mathcal N=2$, $R^{\mathfrak{sl}(N,\mathbb C)[t]}$ is generated by the eight elements of (\ref{equ:inv}), so Theorem \ref{thm:gch} is a corollary of Theorem \ref{thm:chiral}.

Let $X$ be an Enriques surface. Then $X$ is the quotient of a K3 surface $\tilde X$ by a fixed point-free involution $\sigma$. $\sigma$ induce an automorphism of
$\Omega_{\tilde X}^{ch}(\tilde X)$ which maps $Q(z)$, $L(z)$, $J(z)$, $G(z)$, $B(z)$, $D(z)$, $C(z)$, $E(z)$ to $Q(z)$, $L(z)$, $J(z)$, $G(z)$, $-B(z)$, $-D(z)$, $-C(z)$, $-E(z)$, respectively.
Obviously, $\Omega_{X}^{ch}(X)$ is isomorphic to the subalgebra of $\Omega_{\tilde X}^{ch}(\tilde X)^\sigma$ consisting the invariant elements of $\sigma$.{\small }
\begin{thm}
$$\Omega_{X}^{ch}(X)=\Omega_{\tilde X}^{ch}(\tilde X)^\sigma.$$
\end{thm}


\begin{thebibliography}{ABKS}
\bibitem{EHKZ}J. Ekstrand, R. Heluani, J. K\"all\'en, and M. Zabzine, \textit{Chiral de
Rham complex on Riemannian manifolds and special holonomy}.
Comm. Math. Phys. 318 (2013), no. 3, 575-613.
\bibitem{EM} L. Ein and M. Mustata, \textit{Jet schemes and singularities}, Algebraic geometry---Seattle 2005. Part 2, 504--546, Proc. Sympos. Pure Math., 80, Part 2, Amer. Math. Soc., Providence, RI, 2009.
 \bibitem{J}D.D.~Joyce, \textit{Compact Manifolds With Special Holonomy}, Oxford Mathematical Monographs. Oxford University Press, Oxford, 2000. xii+436 pp.
\bibitem{K} V.~Kac, \textit{Vertex Algebras for Beginners}. University Lecture Series 10. Providence, RI: American
 Mathematical Society, 1998.
\bibitem{Ka}M.~Kapranov, \textit{Rozansky-Witten invariants via Atiyah classes}, Compositio. Math. 115(1999) 71-113.
 \bibitem{Kap} A.~Kapustin, \textit{Chiral de Rham complex and the half-twisted sigma-model}. arXiv:hep-th/0504074.

  \bibitem{LSSI}A.~Linshaw, G.~Schwarz, B.~Song, \textit{Jet schemes and invariant theory}, Ann. Inst. Fourier(Grenoble) 65 (2015), no. 6, 2571-2599.
 \bibitem{LSSII}A.~Linshaw, G.~Schwarz, B.~Song, \textit{Arc spaces and the vertex algebra commutant problem}, Adv. Math. 277 (2015), 338-364.
 \bibitem{MSV}F.~Malikov, V.~Schectman, and A.~Vaintrob, \textit{Chiral de Rham complex}, Comm. Math. Phys.  204, (1999) no. 2, 439-473.
 \bibitem{MS}F.~Malikov, V.~Schectman, \textit{Chiral de Rham complex. II},  Differential topology, infinite-dimensional Lie algebras, and applications, 149-188, Amer. Math. Soc. Transl. Ser. 2, 194, Amer. Math. Soc., Providence, RI, 1999.
\bibitem{MS2}F.~Malikov, V.~Schectman, \textit{Chiral Poincar\'e duality,} Math. Res. Lett. 6, (1999) 533-546
 \bibitem{S}B.~Song, \textit{The global sections of the chiral de Rham complex on a Kummer surfaces}, Int. Math. Res. Not. IMRN 2016, no.14, 4271-4296.
 \bibitem{S2}B.~Song, \textit{Chiral Hodge cohomology and Mathieu moonshine}, arXiv:1705.04060 [math.QA].
 \bibitem{We} H. Weyl, \textit{The Classical Groups: Their Invariants and Representations}, Princeton University Press, Princeton, N.J., 1939. xii+302 pp.
 \bibitem{YB}K.~Yano and S.~Bochner,\textit{ Curvature and Betti Numbers}, Annals of Mathemathics Studies, No 32. Princeton University Press, Princeton, N. J., 1953. ix+190 pp.

\end{thebibliography}
\end{document}